\documentclass[12pt]{amsart}

\usepackage{fullpage}
\usepackage{amssymb}
\usepackage{amsmath}
\usepackage[all]{xy}
\usepackage{pdfsync}

\usepackage{hyperref}
\newcommand{\mathlarger}{}

\newtheorem{theorem}{Theorem}[section]
\newtheorem{lemma}[theorem]{Lemma}
\newtheorem{corollary}[theorem]{Corollary}
\newtheorem{proposition}[theorem]{Proposition}

\theoremstyle{definition}
\newtheorem{definition}[theorem]{Definition}
\newtheorem{example}[theorem]{Example}
\theoremstyle{remark}

\newcommand{\isom}{\cong}
\newcommand{\norm}[1]{\lVert#1\rVert}
\newcommand{\abs}[1]{\lvert#1\rvert}
\newcommand{\xto}{\xrightarrow}
\newcommand{\st}{\ | \ }
\newcommand{\eps}{\varepsilon}
\newcommand{\To}{\Rightarrow}  
\newcommand{\incl}{\hookrightarrow}
\newcommand{\from}{\leftarrow}

\newcommand{\R}{\mathbb{R}}

\newcommand{\Z}{\mathbb{Z}}
\newcommand{\F}{\mathbb{F}}

\newcommand{\cat}[1]{\mathbf{#1}}
\renewcommand{\Vec}{\cat{Vec}}
\newcommand{\Top}{\cat{Top}} 
\newcommand{\Cat}{\cat{Cat}}
\newcommand{\Pair}{\cat{Pair}}
\newcommand{\A}{\cat{A}}
\newcommand{\C}{\cat{C}}
\newcommand{\D}{\cat{D}}
\newcommand{\E}{\cat{E}}
\newcommand{\RVec}{{\Vec^{\cat{(\R,\leq)}}}}
\newcommand{\ZVec}{{\Vec^{\cat{(\Z,\leq)}}}}
\newcommand{\RTop}{{\Top^{\cat{(\R,\leq)}}}}
\newcommand{\RPair}{\Pair^{\cat{(\R,\leq)}}}
\newcommand{\RD}{{\D^{\cat{(\R,\leq)}}}}
\newcommand{\RA}{{\A^{\cat{(\R,\leq)}}}}
\newcommand{\CD}{\D^{\C}}
\newcommand{\B}{\mathcal{B}}

\newcommand{\Inter}[2][\eps]{\cat{Int_{#1}({#2})}}
\newcommand{\coker}{\operatorname{coker}}

\DeclareMathOperator{\im}{im}

\DeclareMathOperator{\Id}{Id}
\DeclareMathOperator{\dom}{dom}

\title{Categorification of persistent homology}
\author{Peter Bubenik}
\thanks{The first author gratefully acknowledges support from AFOSR grant \# FA9550-13-1-0115.}
\author{Jonathan A. Scott}
\address{Department of Mathematics, Cleveland State University, 2121 Euclid Ave. RT 1515, Cleveland, OH 44115
}
\email{p.bubenik@csuohio.edu}
\email{j.a.scott3@csuohio.edu}

\keywords{Applied topology, persistent topology, topological persistence, diagrams indexed by the poset of real numbers, interleaving distance}

\subjclass[2010]{55N99, 68W30, 18A25, 18E10, 54E35}

\begin{document}

\begin{abstract}
  We redevelop persistent homology (topological persistence) from a categorical point of view. The main objects of study are $\cat{(\R,\leq)}$-indexed diagrams in some target category. A set of such diagrams has an \emph{interleaving} distance, which we show generalizes the previously-studied bottleneck distance. 
To illustrate the utility of this approach, we generalize previous stability results for persistence, extended persistence, and kernel, image and cokernel persistence.
We give a natural construction of a category of $\eps$-interleavings of $\cat{(\R,\leq)}$-indexed diagrams in some target category, and show that if the target category is abelian, so is this category of interleavings.
\end{abstract}

\maketitle

\section{Introduction}
\label{sec:intro}

The ideas of topological persistence~\cite{elz:tPaS} and persistent homology~\cite{zomorodianCarlsson:computingPH} have had a great impact on computational geometry and the newer field of applied topology.
This method applies geometric and algebraic constructions to input from applications, followed by clever modifications of tools from algebraic topology. It has found many uses, and the results can be global qualitative descriptions inaccessible to other methods. 
Subsequent theoretical work in this subject has given stronger results, and adapted the basic constructions so they might be applied in more diverse situations.
For surveys and books on this subject see \cite{ghrist:survey, edelsbrunnerHarer:survey,carlsson:topologyAndData,edelsbrunnerHarer:book, zomorodian:book}.

\subsection{Motivation}

Throughout its history, algebraic topology has frequently undergone a process in which previous results were redeveloped from a more abstract point of view. 
This has had two main advantages.
First,  abstraction clarified the key ideas and proofs.
Second, and more importantly, the more abstract setting allowed previous results to be vastly generalized and applied in ways never considered in the original. 
The development and use of category theory has been a critical part of this process.

The main motivation of this paper is to subject the ideas and results of topological persistence to this process.

\subsection{Prior work}

In the descriptions below, we will make anachronistic use of this paper's point of view, in particular, its focus on diagrams (see \eqref{eq:Rdiagram} and Section~\ref{sec:categoryTheory}).

Two foundational papers in this subject are \cite{elz:tPaS} and \cite{zomorodianCarlsson:computingPH}.
In the first, Edelsbrunner, Letscher and Zomorodian define persistent homology
for $(\Z_+,\leq)$-indexed diagrams of finite dimensional vector spaces, that are obtained from 
filtered finite simplicial complexes by taking simplicial homology with coefficients in a field.
In the second, Zomorodian and Carlsson take a purely algebraic point of view.
They define persistent homology for tame $(\Z_+,\leq)$-indexed diagrams of finite-dimensional vector spaces, and prove a bijection between isomorphism classes of such tame diagrams and finite barcodes whose endpoints lie in $\Z_+ \cup \{\infty\}$.
$(\Z_+,\leq)$-indexed diagrams of finite dimensional vector spaces are called {persistence modules}.

These papers are rounded out by \cite{cseh:stability}, where Cohen-Steiner, Edelsbrunner and Harer prove that persistent homology is useful in applications by showing that it is stable in the following sense.
Let $f,g:X \to \R$ be continuous functions on a triangulable space.
Define an $\cat{(\R,\leq)}$-indexed diagram of topological spaces, $F$, by setting $F(a) = f^{-1}(-\infty,a]$ and letting $F(a\leq b)$ be given by inclusion.
Define $G$ similarly using $g$.
Let $H$ be the singular homology functor with coefficients in a field.
Assume that $HF$ and $HG$ are diagrams of finite dimensional vector spaces and that they are tame. 
Then the bottleneck distance between $HF$ and $HG$ is bounded by the supremum norm between $f$ and $g$.

This stability result is significantly strengthened by Chazal, Cohen-Steiner, Glisse, Guibas and Oudot in \cite{ccsggo:interleaving}.
They drop the assumptions that $X$ be triangulable, that $f,g$ be continuous, and that $HF$ and $HG$ be tame.
Their approach is crucial to this paper.
They explicitly work with $\cat{(\R,\leq)}$-indexed diagrams, though they consider them from an algebraic, not categorical, point of view.
They define the interleaving distance, $d$, between such diagrams, and define the bottleneck distance, $d_B$, between such diagrams using limits of discretizations, and show that $d_B \leq d$.

The basic idea of persistent homology has been extended in numerous ways. Here we focus on two particularly useful extensions, given in \cite{cseh:extendingP} and \cite{csehm:kernels}.
In the first, Cohen-Steiner, Edelsbrunner and Harer, define extended persistence for finite-dimensional simplicial complexes with a finite filtration and homology with coefficients in $\Z / 2\Z$. They show that in this case the stability result of~\cite{cseh:stability} applies.
In the second, Cohen-Steiner, Edelsbrunner, Harer and Morozov consider a triangulated space $X$ with subcomplex $Y$ and maps $f,f':X \to \R$ and $g,g':Y \to \R$ such that for all $y \in Y$, $f(y)\leq g(y)$ and $f'(y) \leq g'(y)$. 
They assume that $f,f',g,g'$ are continuous and tame.
Then there are maps of the corresponding $\cat{(\R,\leq)}$-indexed diagrams $HG \to HF$ and $HG' \to HF'$.
Let $\eps = \max\{\norm{f-f'}_{\infty},\norm{g-g'}_{\infty}\}$.
The authors show that the bottleneck distances between the kernels, images and cokernels, respectively, of these maps, are each bounded above by $\eps$.

An early categorical approach to persistence can be found in~\cite{cfp:sizeCategorical}.

\subsection{Our contributions}

We redevelop persistent homology from a categorical point of view.
In particular, we consider diagrams indexed by $\cat{(\R,\leq)}$ to be the main objects of study. 
An $\cat{(\R,\leq)}$-indexed diagram consists of a set of objects $X(a)$ for each $a \in \R$ and morphisms 
\begin{equation} \label{eq:Rdiagram} 
X(a) \to X(b)
\end{equation}
for each $a\leq b$, satisfying certain composition and unit axioms (see Section~\ref{sec:categoryTheory}).
The objects and morphisms lie in some fixed category, such as topological spaces and continuous maps, or finite-dimensional vector spaces and linear transformations.
In Section~\ref{sec:categoricalPH}, we show that the basic constructions of persistent homology are special cases of this construction. 
We will show that in this setting, functoriality provides concise and powerful results.

In Section~\ref{sec:interleaving},
we define an $\eps$-interleaving for $\cat{(\R,\leq)}$-indexed diagrams (Definition~\ref{def:interleaving}) and show that this induces a metric (Theorem~\ref{thm:metric} and Corollary~\ref{cor:metric}).

We specialize to $\cat{(\R,\leq)}$-indexed diagrams of finite-dimensional vector spaces in Section~\ref{sec:vec}.
These are also called \emph{(real) persistence modules}.
We study barcodes, persistence diagrams, and the bottleneck and interleaving distances.
We define finite type diagrams to be direct sums of certain indecomposable diagrams (Definition~\ref{def:finiteType}).
We show that these are exactly the tame diagrams (Theorem~\ref{thm:tameFiniteType}).
Furthermore, we show that they satisfy a Krull-(Remak-)Schmidt theorem. That is, the direct sum decomposition is essentially unique (Corollary~\ref{cor:krullSchmidt}).
We show that the metric space of finite barcodes together with the bottleneck distance embeds isometrically into the metric space of $\cat{(\R,\leq)}$-indexed diagrams of finite-dimensional vector spaces with the interleaving distance
(Theorem~\ref{thm:embedding}).
This result justifies our assertion that our stability theorems, which use the interleaving distance, are generalizations of previously established stability theorems, which use the bottleneck distance.

In Section~\ref{sec:stability}, we give a simple formal argument for a stability theorem for the interleaving distance.
By the previous work identifying the interleaving and bottleneck distances, this
allows us to both remove assumptions, and to significantly generalize,
the stability result of~\cite{cseh:stability}. 
Given any functions $f,g:X \to \R$ on any topological space $X$ and any functor $H$ on topological spaces, we show that the interleaving distance of $HF$ and $HG$ is bounded above by the supremum norm between $f$ and $g$ (Theorem~\ref{thm:stability}).

We generalize the extended persistence construction of~\cite{cseh:extendingP} in Section~\ref{sec:extended}. For any (not necessarily continuous) map $f:X \to (-\infty,M] \subset \R$, we define a $\cat{(\R,\leq)}$-indexed diagram of pairs of topological spaces. We prove a stability theorem for extended persistence. Given $f,g:X \to (-\infty,M]$ and corresponding diagrams $F$ and $G$ of pairs of spaces, and any functor $H$ on pairs of spaces, the interleaving distance between $HF$ and $HG$ is bounded above by the supremum norm between $f$ and $g$ (Theorem~\ref{thm:stability-extended}).

In Section~\ref{sec:abelian}, we define a category of interleavings of $\cat{(\R,\leq)}$-indexed diagrams in a given base category (Definition~\ref{def:Inter}). 
We show that in the case that the base category is an abelian category, then so is this category of interleavings (Theorem~\ref{thm:inter-abcat}).
As a result, this category has direct sums, kernels, images and cokernels.
As an application, we generalize the stability theorem of~\cite{csehm:kernels}, dropping the assumptions that $X$ and $Y$ are triangulated, that $f,f',g,g'$ are continuous and tame, replacing the subcomplex condition with a continuous map $Y \to X$ and replacing singular homology with coefficients in $\Z/2\Z$, with any functor from topological spaces to an abelian category (Theorem~\ref{thm:stability-kernels}).
We also give a version of this theorem for extended persistence (Theorem~\ref{thm:stability-kernels-extended}).

\subsection{Comparison with other recent work}
\label{sec:recent}

The material in Section~\ref{sec:vec} has been studied in greater
detail in the algebraic setting by Lesnick~\cite{lesnick:thesis} and by
Chazal, de Silva, Glisse and Oudot~\cite{csgo:persistenceModules}. In
particular, Lesnick proves a more general Isometry
Theorem~\cite[Theorem 2.4.2]{lesnick:thesis}, removing the condition
that the persistence modules have finite type from our
Theorem~\ref{thm:embedding}.  This is further generalized to q-tame
persistence modules in~\cite[Theorem 4.11]{csgo:persistenceModules}.
We also remark that one of the directions in our isometry theorem is
due to~\cite{ccsggo:interleaving}.
Crawley-Boevey~\cite{crawley-boevey} has shown that any
$(\R,\leq)$-indexed diagram of finite-dimensional vector spaces is a
direct sum of interval modules. In light of this result, it may be
possible to generalize some of our work in Section~\ref{sec:vec}.

Our Stability Theorem (Theorem~\ref{thm:stability}) is quite general
and once the categorical machinery has been set up, has a very simple
proof. However it applies to persistence modules, not for their
corresponding persistence diagrams. In the language of~\cite{bdss:1},
it is a soft stability theorem.  Hard stability
theorems~\cite{cseh:stability,ccsggo:interleaving,csgo:persistenceModules}
giving stability for persistence diagrams require more detailed
analysis. For example, an Isometry Theorem can be used to show that
soft stability implies hard stability. On the other hand our Stability
Theorem is more general in that it applies to functors to arbitrary
categories. For a simple example, consider homology with integer
coefficients.  It also clarifies what part of stability is purely formal
and what part requires detailed analysis. This viewpoint is expanded
upon in~\cite{bdss:1}.

\section{Background}
\label{sec:background}

In Section~\ref{sec:categoryTheory}, we give the basic definitions of category theory that we will use throughout the paper. In Section~\ref{sec:categoricalPH}, we show how the standard constructions of persistent homology fit within our categorical approach. 
The last two sections give more specialized background.
In Section~\ref{sec:abelianCatDefs}, we define abelian categories, which we use in Section~\ref{sec:abelian}.
In Section~\ref{sec:algebra}, we give some algebraic definitions used in the proof of Theorem~\ref{thm:tameFiniteType}.

\subsection{Categorical terminology}
\label{sec:categoryTheory}

A \emph{category}, $\cat{C}$, consists of a class of objects, $\cat{C}_{0}$, and for each pair of objects $X,Y \in \cat{C}_{0}$, a set of \emph{morphisms}, $\cat{C}(X,Y)$.  We often write $f : X \rightarrow Y$ if $f \in \cat{C}(X,Y)$.  For every triple $X,Y,Z \in \cat{C}_{0}$, there is a set mapping,
\[
	\cat{C}(Y,Z) \times \cat{C}(X,Y) \rightarrow \cat{C}(X,Z), \quad (g,f) \mapsto gf
\]
called \emph{composition}.  Composition must be \emph{associative}, in the sense that $(hg)f = h(gf)$.  Finally, for all $X \in \cat{C}$, there is an \emph{identity} morphism, $\Id_{X} : X \rightarrow X$, that satisfies $\Id_{X} f = f$ and $g \Id_{X} = g$ for all $f : W \rightarrow X$ and all $g : X \rightarrow Y$.  The identity morphism is unique.  We will regularly abuse notation and write $X \in \cat{C}$ to mean $X \in \cat{C}_{0}$.

A category $\cat{C}$ is called \emph{small} if $\cat{C}_{0}$ is a set rather than a proper class.  

\begin{example}
Let $\Top$ be the category whose objects are all topological spaces, and whose morphisms are all continuous maps.  Here, composition is the composition of mappings, and the identity morphisms are what one would expect.

A related category is $\Pair$, whose objects are pairs $(X,A)$, where $X$ is a topological space and $A$ is a  subspace of $X$.  A morphism from $(X,A)$ to $(Y,B)$ is a continuous map $f : X \rightarrow Y$ such that $f(A) \subset B$.  We express this condition by saying that the diagram 
\[
	\xymatrix{
		A	\ar[r]^{f|_{A}} \ar[d]_{j_{A}}	& B 	\ar[d]^{j_{B}}	\\
		X	\ar[r]_{f}				& Y
	}
\]
commutes, where $j_{A}$ and $j_{B}$ are the canonical inclusions, and $f|_{A}$ is $f$ restricted to $A$.
\end{example}

\begin{example}
Let $\Vec$ be the category of finite-dimensional vector spaces over a fixed ground field $\F$, along with the linear transformations between them.  Again, composition is that of mappings, and the identities are simply the identity mappings.

A \emph{graded vector space} is a collection $V_{*} = \{ V_{n} \}_{n \in \Z}$, with each $V_{n} \in \Vec$.  A morphism, $f_{*} : V_{*} \rightarrow W_{*}$, of graded vector spaces is a sequence, $f_{*} = \{ f_{n} : V_{n} \rightarrow W_{n} \}$.
Denote by $\cat{grVec}$ the category of graded vector spaces and their morphisms.
\end{example}

A reflexive, antisymmetric, and transitive relation $\leq$ on a set $P$ is called a \emph{partial order}. A set $P$ equipped with a partial order is called a \emph{poset}.   We identify each poset $P$ with the small category $\cat{P}$ that has $\cat{P}_{0} = P$, and $\cat{P}(x,y)$ has precisely one element if $x \leq y$ and is otherwise empty.  Conversely, let $\cat{P}$ be a small category in which each set of morphisms contains at most one element, and if $\cat{P}(x,y)$ and $\cat{P}(y,x)$ are both nonempty, then $x=y$.  Then $\cat{P}_{0}$ is a poset, with partial ordering defined by $x \leq y$ if and only if $\cat{P}(x,y) \neq \varnothing$. 

\begin{example}
The set of real numbers, $\R$, with its usual ordering, is a poset.  The set of integers, $\Z$, of non-negative integers $\Z_{+}$, and $\cat{[n]} = \{ 0, \ldots, n \}$, are sub-posets.
For a partial order that is not a total order, consider the set $\R^n$ with $n>1$ and the ordering $(x_1,\ldots,x_n) \leq (y_1,\ldots,y_n)$ if and only if $x_i \leq y_i$ for all $i = 1, \ldots, n$.
\end{example}

Two objects $X,Y \in \cat{C}_{0}$ are said to be \emph{isomorphic} if there exist morphisms $f : X \rightarrow Y$ and $g : Y \rightarrow X$ such that $g f = \Id_{X}$ and $f g = \Id_{Y}$.  In this case, $f$ and $g$ are called \emph{isomorphisms}.  Clearly, isomorphism is an equivalence relation.  In $\Top$, isomorphism becomes \emph{homeomorphism}.

The notion of \emph{functor} expresses relationships between categories.  Let $\cat{A}$ and $\cat{C}$ be categories.  A \emph{functor}, $F : \cat{A} \rightarrow \cat{C}$, consists of a mapping $F:\cat{A}_{0} \rightarrow \cat{C}_{0}$, and for each pair $X,Y \in \cat{A}_{0}$, a mapping $F : \cat{A}(X,Y) \rightarrow \cat{C}(F(X), F(Y))$.  These mappings must be compatible with the composition and identity structure of the categories, in the sense that if $f : X \rightarrow Y$ and $g : Y \rightarrow Z$, then $F(gf) = F(g)F(f)$, and if $X \in\cat{A}_0$, then $F(\Id_{X}) = \Id_{F(X)}$.

\begin{example}
Denote by $H_{*}(-)$ singular homology with coefficients in some fixed field, $\F$.  Then $H_{*}(X)$ is a graded $\F$-vector space for all $X \in \Top_{0}$.  Furthermore, if $f : X \rightarrow Y$ is continuous, then we get the induced homomorphism, $H_{*}(f) : H_{*}(X) \rightarrow H_{*}(Y)$.  Since $H_{*}(gf) = H_{*}(g)H_{*}(f)$, singular homology defines a functor $H_{*} : \Top \rightarrow \cat{grVec}$.
If we consider only homology in degree $k$, then we get a functor $H_k:\Top \to \Vec$.
\end{example}

Let $F,G : \cat{A} \rightarrow \cat{C}$ be functors. A \emph{natural transformation}  $\eta : F \Rightarrow G$ consists of, for all $A \in \cat{A}_{0}$, a morphism $\eta_{A} : F(A) \rightarrow G(A)$ in $\cat{C}$, such that whenever $\varphi : A \rightarrow A'$ is a morphism, the diagram
\begin{equation} \label{eq:naturalTransformation}
	\xymatrix{
		F(A) \ar[r]^{\eta_{A}} \ar[d]_{F(\varphi)} & G(A) \ar[d]^{G(\varphi)}	\\
		F(A')	\ar[r]_{\eta_{A'}} 				& G(A')
	}
\end{equation}
commutes.
If for all $A \in \cat{A}_0$, $\eta_A$ is an isomorphism, then $\eta$ is called a \emph{natural isomorphism} and we write $F \isom G$.

\begin{example}
Consider the poset $(\R,\leq)$, and let $\eps \geq 0$.  Define $T_{\eps}:(\R, \leq) \rightarrow (\R,\leq)$ by $T_{\eps}(x) = x + \eps$.  If $x \leq y$ then $x + \eps \leq y + \eps$, so $T_{\eps}$ defines a functor to $(\R, \leq)$ to itself.  We call $T_{\eps}$ \emph{translation by $\eps$}. Since $\eps \geq 0$, $x \leq x + \eps$ for all $x \in \R$, so we get a natural transformation $\eta : I \Rightarrow T_{\eps}$, where $I : \R \rightarrow \R$ is the identity functor.
\end{example}

The collection of all small categories, and the functors between them, itself forms a category, denoted by $\Cat$.

Let $\cat{C}$ and $\cat{D}$ be categories with $\cat{C}$ small.  A functor, $F : \cat{C} \rightarrow \cat{D}$, is called a \emph{diagram in $\cat{D}$ indexed by $\cat{C}$}.  The collection of all such functors, and natural transformations between them, forms a category, $\CD$.  

\begin{example}
Let $\cat{C}$ be the discrete category whose objects are the integers; the only morphisms are the identity morphisms.  Then $\Vec^{\cat{C}} = \cat{grVec}$.
\end{example}

\begin{example}
A diagram $F$ in a category $\cat{D}$ indexed by $(\Z_{+}, \leq)$ is a sequence of morphisms in $\cat{D}$:
\[
	F(0) \rightarrow F(1) \rightarrow F(2) \rightarrow \cdots.
\]
If $\cat{D} = \Top$ then each $F(n)$ is a topological space 
and the morphisms are continuous maps.
If $\cat{D} = \Vec$ then each $F(n)$ is a finite-dimensional vector space
and the morphisms are linear maps.

Indexed by $(\Z,\leq)$, the diagram extends in both directions:
\[
	\cdots \rightarrow F(-2) \rightarrow F(-1) \rightarrow F(0) \rightarrow F(1) \rightarrow F(2) \rightarrow \cdots.
\]
If the indexing category is $(\R,\leq)$, then we have objects $F(a)$ for all $a \in \R$, and for each $a \leq b$, a morphism $F(a) \rightarrow F(b)$.
\end{example}

Given two natural transformations $\varphi: F \Rightarrow G$ and $\psi: G \Rightarrow H$, their (vertical) composition $\psi \circ \varphi$ is the natural transformation given by the composition of morphisms $F(A) \xto{\varphi_A} G(A) \xto{\psi_A} H(A)$ and the composition of the corresponding commutative squares~\eqref{eq:naturalTransformation}.

For $i=1,2$, let $F_{i},G_{i}:\cat{A}_{i} \rightarrow \cat{A}_{i+1}$ be functors, and let $\varphi_{i} : F_{i} \Rightarrow G_{i}$ be a natural transformation.  The (horizontal) composition of $\varphi_{1}$ and $\varphi_{2}$ is the natural transformation, $\varphi_{2}\varphi_{1}:F_{2}F_{1} \Rightarrow G_{2}G_{1}$, defined on morphisms by
$(\varphi_{2}\varphi_{1})(f) = \varphi_{2}(\varphi_{1}(f))$.
For every functor $H$, there is the identity natural isomorphism $\Id_H: H \Rightarrow H$. We abuse notation and refer to the horizontal composition of a natural transformation $\varphi$ with $\Id_H$ as the composition of $\varphi$ with $H$.

\subsection{Categorical persistent homology}
\label{sec:categoricalPH}

In this section we consider two prototypical examples in which persistent homology is applied and show how they fit into our categorical framework.
We also show how diagrams indexed by $\cat{[n]}$, $\cat{(\Z_+,\leq)}$, and $\cat{(\Z,\leq)}$ are special cases of diagrams indexed by $\cat{(\R,\leq)}$.
Finally, we define persistent homology.

\subsubsection{Filtered simplicial complexes}

First, let $K$ be a finite simplicial complex with filtration 
\begin{equation*}
  \emptyset = K_0 \subseteq K_1 \subseteq \cdots \subseteq K_n = K.
\end{equation*}
Then this gives an $\cat{[n]}$-indexed diagram of topological spaces, i.e. $K \in \Top^{\cat{[n]}}$, with $K(i) = K_i$ and $K(i\leq j)$ given by inclusion.

Let $H_k$ be the degree $k$ simplicial homology functor with coefficients in a field $\F$.
Then $H_k K$ is an $\cat{[n]}$-indexed diagram of finite dimensional vector spaces.
That is, $H_k K(i) = H_k(K_i,\F)$ and $H_k K(i\leq j)$ is the map induced on homology by the inclusion $K_i \incl K_j$. So $H_k K \in \Vec^{\cat{[n]}}$.

We can sum homology in all degrees to get $HF \in \Vec^{\cat{[n]}}$, given by $HF(i) = \oplus_k H_k(K_i,\F)$.

\subsubsection{Sublevel sets}

Second, let $X$ be a topological space, and let $f: X \to \R$ be a not necessarily continuous real-valued function on $X$. Let $a \in \R$.
We consider the \emph{sublevel set} (or \emph{lower excursion set}, also called a \emph{half space}) 
\begin{equation*}
  f^{-1}((-\infty,a]) = \left\{ x \in X \st f(x) \leq a \right\}.
\end{equation*}
For simplicity, we will usually write $f^{-1}(-\infty,a]$.
We consider it as a topological space using the subspace topology.
Notice that if $a \leq b$ then $f^{-1}(-\infty,a] \subseteq f^{-1}(-\infty,b]$, and this inclusion is a continuous map.

This data can be assembled into an $\cat{(\R,\leq)}$-indexed diagram of topological spaces, $F \in \RTop$. For $a \in \R$, we define $F(a) = f^{-1}(-\infty,a]$. For $a \leq b$, we define $F(a\leq b)$ to be the inclusion $f^{-1}(-\infty,a] \incl f^{-1}(-\infty,b]$. It is easy to check that this defines a functor $F: \cat{(\R,\leq)} \to \Top$.

Let $H_k$ be the $k$th singular homology functor with coefficients in some field $\F$.
Then $H_kF$ is an $\cat{(\R,\leq)}$-indexed diagram of (not necessarily finite dimensional) vector spaces. That is, $H_kF(a) = H_k(f^{-1}(-\infty,a], \F)$, and for $a \leq b$, $H_kF(a\leq b)$ is the map induced on homology by the inclusion
$f^{-1}(-\infty,a] \incl f^{-1}(-\infty,b]$.
If $f$ has the property that for all $a \in \R$, $H_k(f^{-1}(-\infty,a], \F)$ is a finite dimensional vector space, then $H_kF \in \RVec$.

If $f$ has the property that for all $a \in \R$, $H_*(f^{-1}(\infty,a],\F)$ is finite-dimensional, then $HF \in \RVec$ is given by $HF(a) = \oplus_k H_k(f^{-1}(-\infty,a],\F)$.

\subsubsection{Diagrams by $\cat{[n]}$, $\cat{(\Z_+,\leq)}$, and $\cat{(\Z,\leq)}$}

In this paper we will only consider the indexing category $\cat{(\R,\leq)}$. However, this case also includes the cases $\cat{[n]}$, $\cat{(\Z_+,\leq)}$ and $\cat{(\Z,\leq)}$, by the following observation. 
Consider $F \in \Top^{\cat{[n]}}$. Then we can extend $F$ to an $\cat{(\R,\leq)}$-indexed diagram as follows.
The inclusion functor $\cat{i: [n] \to \cat{(\R,\leq)}}$ given by $\cat{i}(j) = j$ has a retraction functor $\cat{r: (\R,\leq) \to [n]}$ given by 
\begin{equation*}
\cat{r}(a) =
\begin{cases}
  0 & \text{ if } a \leq 0, \\ 
  \lfloor a \rfloor & \text{ if } 0 <  a < n\\
  n & \text{ if } a \geq n.
\end{cases}
\end{equation*}
Thus the composite functor $F\cat{r}$ is an element of $\RTop$, and $F\cat{ri} = F$.
There are similarly defined retraction functors to $\cat{(\Z_+,\leq)}$ and to $\cat{(\Z,\leq)}$.

\subsubsection{Persistent homology}

Given a diagram $F \in \RTop$, we define the \emph{$p$-persistent $k$th homology group of $F(a)$} to be the image of the map $H_kF(a \leq a+p)$.

\subsubsection{Persistence modules}
\label{sec:persistence-modules}

Diagrams in $\Vec^\cat{[n]}$, $\Vec^{(\mathbb{Z_+},\leq)}$ and $\RVec$ are often called \emph{persistence modules}.

\subsection{Abelian categories}
\label{sec:abelianCatDefs}

In this section we recall standard definitions from category theory that we will use in Section~\ref{sec:abelian}.  Details can be found in, for example, \cite{maclane:book}.   Throughout this section,  $\cat{C}$ denotes a category.

\subsubsection{Initial, Terminal, and Final Objects} We say that an object $\varnothing$ of $\cat{C}$ is \emph{initial} if, for every object $X$ in $\cat{C}$, there is a unique morphism $\varnothing \rightarrow X$.  An object $*$ is \emph{terminal} if, for every object $X$, there is a unique morphism $X \rightarrow *$.  It follows from these definitions that initial and terminal objects, if they exist, are unique up to canonical isomorphism.  If an object is both initial and terminal, we say that it is \emph{zero}, and denote it by $0$.  In the presence of a zero object, for every pair of objects $X,Y \in \cat{C}$, we can define the \emph{zero morphism} $0 : X \rightarrow Y$ to be the composite of the unique morphisms, $X\rightarrow 0 \rightarrow Y$.  It follows by uniqueness that if $f$ is any morphism, then $f0 = 0f = 0$.

\subsubsection{Monomorphisms, epimorphisms, kernels and cokernels}
Let $f : X \rightarrow Y$ be a morphism.  We say that $f$ is a \emph{monomorphism} if, whenever $g,h : W \rightarrow X$ are morphisms such that $fg=fh$, we have that $g=h$.  Dually, $f$ is an \emph{epimorphism} if, whenever $k,\ell : Y \rightarrow Z$ are morphisms such that $k f = \ell f$, then $k = \ell$.
An isomorphism class of monomorphisms to $Y$ is called a \emph{subobject} of $Y$.
Dually, isomorphism classes of epimorphisms are called \emph{quotient objects}.

Suppose that $\cat{C}$ has a zero object, $0$.  Let $f : X \rightarrow Y$ be a morphism in $\cat{C}$.  The \emph{kernel} of $f$ is the equalizer of $f$ and $0 : X \rightarrow Y$. That is, the kernel is a  morphism, $j : \ker f \rightarrow X$, such that $fj=0$ and that is ``universal'' in the sense that whenever $g : W \rightarrow X$ is a morphism satisfying $fg=0$, then there is a unique morphism $\tilde{g} : W \rightarrow \ker f$ such that $j \tilde{g} = g$. 
Since $j$ is an equalizer, it follows that $j$ is a monomorphism.
So $\ker f$ represents a subobject of $X$.
Thus the kernel is the appropriate categorical notion for the part of $X$ that $f$ sends to 0.
We use the word ``kernel'' to mean both the object $\ker f$ and the universal morphism, $\ker f \rightarrow X$, according to the context. 
We remark that it follows from the definition that all such universal objects are unique up to unique isomorphism. That is, if $g:W \to X$ and $g':W' \to X$ are both kernels of $f:X \to Y$ then there is a unique isomorphism $\tilde{g}:W \to W'$ such that $g'\tilde{g} = g$.

Dually, the \emph{cokernel} of $f : X \rightarrow Y$ is the coequalizer of $f$ and $0$. That is, the cokernel is a universal morphism $q : Y \rightarrow \coker f$, such that whenever $h : Y \rightarrow Z$ satisfies $hf=0$, there exists a unique morphism $\tilde{h} : \coker f \rightarrow Z$ such that $\tilde{h}q = h$.  Again, we sometimes abuse notation and use ``cokernel'' to refer to the object, $\coker f$.
Since $q$ is a coequalizer, it is an epimorphism, and $\coker f$ represents a quotient object of $Y$.
Again, cokernels, if they exist, are unique up to canonical isomorphism.

\subsubsection{Products, coproducts, pull-backs and push-outs}

Let $X,Y \in \cat{C}$.  The \emph{product} of $X$ and $Y$, if it exists in $\cat{C}$, is an object denoted by $X \times Y$, along with morphisms $p_{X} : X \times Y \rightarrow X$ and $p_{Y} : X \times Y \rightarrow Y$ satisfying the following universal property.
For every object $W$ together with a pair of morphisms $f_{X} : W \rightarrow X$ and $f_{Y} : W \rightarrow Y$, there is a unique morphism $f : W \rightarrow X \times Y$ such that $f_{X} = p_{X} f$ and $f_{Y} = p_{Y} f$.  The product, if it exists, is unique up to canonical isomorphism.

Dually, the \emph{coproduct} of $X$ and $Y$, if it exists in $\cat{C}$, is an object $X \oplus Y$, along with morphisms $j_{X} : X \rightarrow X \oplus Y$ and $j_{Y} : Y \rightarrow X \oplus Y$ satisfying the following universal property.
For every object $U$ together with a pair of morphisms $g_{X} : X \rightarrow U$ and $g_{Y} : Y \rightarrow U$, there is a unique morphism $g : X \oplus Y \rightarrow U$ such that $g_{X} = g j_{X}$ and $g_{Y} = g j_{Y}$.  The coproduct, if it exists, is unique up to canonical isomorphism.

Consider the diagram $X \xrightarrow{f} Z \xleftarrow{g} Y$.  The \emph{pull-back} of $f$ and $g$ consists of an object $P$, and morphisms $X \xleftarrow{p_{X}} P \xrightarrow{p_{Y}}Y$ satisfying $f p_{X} = g p_{Y}$ and the following universal property. For each diagram
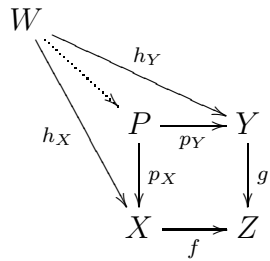
\[
	\xymatrix{
		W \ar[rrd]^{h_{Y}} \ar[ddr]_{h_{X}} \ar@{.>}[dr]	\\
			& P \ar[r]_{p_{Y}} \ar[d]^{p_{X}} 	& Y \ar[d]^{g}	\\
			& X \ar[r]_{f} 		& Z  
	}
\]
where the outer paths commutes, there is a unique morphism $W \rightarrow P$ that makes the entire diagram commute.  The pull-back is unique up to canonical isomorphism, and is denoted by $P = X \times_{Z} Y$ when there can be no ambiguity concerning $f$ and $g$.

Dually, the \emph{push-out} of the diagram $Y \xleftarrow{f} X \xrightarrow{g} Z$ consists of an object $Q$ along with universal morphisms $Y \xrightarrow{j_{Y}} Q \xleftarrow{j_{Z}} Z$, satisfying $j_{Y} f = j_{Z} g$ and the following universal property. Whenever the outer paths in the diagram
\[
	\xymatrix{
		X \ar[r]^{g} \ar[d]_{f}			& Z \ar[d]_{j_{Z}} \ar[ddr]^{k_{Z}}	 \\
		Y \ar[r]^{j_{Y}} \ar[drr]_{k_{Y}}		& Q \ar@{.>}[dr] 				\\
								&				& U
	}
\]
commute, there is a unique morphism $k:Q \rightarrow U$ making the entire diagram commute.  The push-out is unique up to canonical isomorphism, and is denoted by $Q = Y \oplus_{X} Z$.

\subsubsection{Abelian categories} \label{subsubsec:abcat}
An \emph{abelian category} is a category that contains a zero object and all products and coproducts, in which every morphism has a kernel and cokernel, every monomorphism is a kernel, and every epimorphism is a cokernel.  By Freyd~\cite{freyd:book-reprinted}, 
every abelian category, $\cat{A}$, is \emph{preadditive}, that is, it is naturally enriched in abelian groups.  This means that for all pairs of objects, $X$ and $Y$, the set of morphisms $\cat{A}(X,Y)$ is an abelian group, and composition is bilinear.
Furthermore, binary products and coproducts coincide, in the sense that the natural morphism $X \oplus Y \rightarrow X \times Y$ is an isomorphism.

We say that an object $X$ of an abelian category is \emph{indecomposable} if whenever $X \cong U \oplus V$, either $U \cong 0$ or $V \cong 0$.

\begin{example}
Let $\Vec$ be the category of finite-dimensional vector spaces over some fixed field, $\F$.  The morphisms are linear transformations.  The zero object is (an element of the isomorphism class of) the trivial vector space, $0 = \{ 0 \}$.  The product of $V$ and $W$ is the Cartesian (direct) product, $V \times W$.  The coproduct is the direct sum, $V \oplus W$, which is canonically isomorphic to the direct product.  

If $f : V \rightarrow W$ is linear, we set $\ker f = \{ v \in V \mid f(v) = 0 \}$, $f(V) = \{ f(v) \mid v \in V \}$, and $\coker f = W / f(V)$.  It is a straightforward exercise to show that monomorphisms are simply injective linear transformations, and that if $f$ is injective, then  $V \cong f(V)$, and so $V$ is (isomorphic to) the kernel of the quotient map, $W \rightarrow \coker f$.  Similarly, epimorphisms are surjective linear transformations, and by the First Homomorphism Theorem, if $f : V \rightarrow W$ is  surjective, then $W$ is the cokernel of $\ker f \rightarrow V$.  Thus, $\Vec$ is an abelian category.
\end{example}


\subsection{Algebra} \label{sec:algebra}
We will need the following definitions in Lemma~\ref{lem:ZVec}, which we use in the proof of Theorem~\ref{thm:tameFiniteType}.

A (non-negatively) \emph{graded ring} is a ring, $R$, along with a direct-sum decomposition, $R = \oplus_{n=0}^{\infty} R_{n}$, such that $1 \in R_{0}$, and if $a \in R_{m}$ and $b \in R_{n}$, then $ab \in R_{m+n}$.  Our primary example  will be the polynomial ring $\F[t]$, for a field $\F$, which is graded by degree.

A \emph{graded $\F[t]$-module} is an $\F[t]$-module, $M$, with a decomposition $M = \oplus_{n=0}^{\infty} M_{n}$, that satisfies $t^{m} x \in M_{m+n}$ whenever $x \in M_{n}$.  We say that $M$ has \emph{finite type} if each $M_{n}$ is finite dimensional over $\F$.

We will also make use of the following structure theorem for finitely generated modules over a principal ideal domain.

\begin{theorem}\label{thm:struct-pid}
\cite[Theorem 6.12(ii),p. 225]{hungerford}
Let $A$ be a finitely generated module over a principal ideal domain $R$.  Then $A$ is the direct sum of a free submodule $E$ of finite rank and a finite number of cyclic torsion modules. The cyclic torsion summands (if any) are of orders $p_{1}^{s_{1}}, \ldots, p_{k}^{s_{k}}$, where $p_{1}, \ldots, p_{k}$ are (not necessarily distinct) positive integers. The rank of $E$ and the list of ideals $(p_{1}^{s_{1}}), \ldots, (p_{k}^{s_{k}})$ are uniquely determined by $A$ (except for the order of the $p_{i}$).
\end{theorem}

\section{Interleavings of diagrams}
\label{sec:interleaving}

In this section we define $\eps$-interleavings for $\cat{(\R,\leq)}$-indexed diagrams and show that they induce a metric on a set of $\cat{(\R,\leq)}$-indexed diagrams.
Our definition is a categorical version of the definition in~\cite{ccsggo:interleaving}.

\medskip

We consider the category $\cat{(\R,\leq)}$, whose objects are the real numbers and the set of morphisms from $a$ to $b$ consists of a single morphism if $a \leq b$ and is otherwise empty. 
For $b \geq 0$,
define $T_b: \cat{(\R,\leq)} \to \cat{(\R,\leq)}$ to be the functor given by $T_b(a) = a+b$,
and 
define $\eta_b:\Id_{\cat{(\R,\leq)}} \To T_b$ to be the natural transformation given by $\eta_b(a): a \leq a+b$.
Note that $T_{b}T_{c} = T_{b+c}$ and that $\eta_{b}\eta_{c} = \eta_{b+c}$.

Let $\D$ be any category and let $\eps\geq 0$.
Let $F, G \in \RD$.

\begin{definition} \label{def:interleaving}
  An \emph{$\eps$-interleaving} of $F$ and $G$ consists of natural transformations $\varphi: F \To G T_{\eps}$ and $\psi: G \To F T_{\eps}$,
i.e.
\begin{equation*}
    \xymatrix{
      \cat{(\R,\leq)} \ar[r]^{T_{\eps}} \ar[d]_F \ar@<1ex>@{}[dr]|*+{\stackrel{\mathlarger{\varphi}}{\Rightarrow}} 
& \cat{(\R,\leq)} \ar[d]_G \ar[r]^{T_{\eps}}
\ar@<1ex>@{}[dr]|*+{\stackrel{\mathlarger{\psi}}{\Rightarrow}}  
& \cat{(\R,\leq)} \ar[d]^F \\ 
   D \ar@{=}[r] & D \ar@{=}[r] & D
    }
\end{equation*}
such that 
\begin{equation} \label{eq:interleaving}
(\psi T_{\eps}) \varphi = F \eta_{2\eps} \text{ and } 
(\varphi T_{\eps}) \psi = G \eta_{2\eps}.
\end{equation}
If $(F,G,\varphi,\psi)$ is an $\eps$-interleaving, then we say that $F$ and $G$ are $\eps$-\emph{interleaved}.
\end{definition}

The existence of the natural transformations $\varphi$ and $\psi$ implies that we have the following commutative diagrams for all $a \leq b$.
\begin{equation*}
\xymatrix@R=1em@C=0.5em{F(a) \ar[rr] \ar[dr]_{\varphi(a)} & & F(b) \ar[dr]^{\varphi(b)}\\
& G(a+\eps) \ar[rr] & & G(b+\eps)
}
\quad
\xymatrix@R=1em@C=0.5em{& F(a+\eps) \ar[rr] & & F(b+\eps)\\
G(a) \ar[ur]^{\psi(a)} \ar[rr] & & G(b) \ar[ur]_{\psi(b)}
}
\end{equation*}
The identities \eqref{eq:interleaving} imply that the following diagrams commute for all $a$.
\begin{equation*}
\xymatrix@R=1em@C=0.5em{F(a) \ar[rr] \ar[dr]_{\varphi(a)} & & F(a+2\eps)\\
& G(a+\eps) \ar[ur]_{\psi(a+\eps)}
}
\quad
\xymatrix@R=1em@C=0.5em{
& F(a+\eps) \ar[dr]^{\varphi(a+\eps)}\\
G(a) \ar[rr] \ar[ur]^{\psi(a)} & & G(a+2\eps)
}
\end{equation*}

\begin{definition} \label{def:interleaving-distance}
Say that $d(F,G) \leq \eps$ if $F$ and $G$ are $\eps$-interleaved.
Explicitly, 
\begin{equation*}
  d(F,G) = \inf \{ \eps \geq 0 \st F \text{ and } G \text{ are } \eps \text{-interleaved} \},
\end{equation*}
where we set $d(F,G) = \infty$ if  $F$ and $G$ are not $\eps$-interleaved for any $\eps \geq 0$.
\end{definition}

We will show that this function $d$ is a generalized metric.
It fails to be a metric because 
it can take the value $\infty$ and $d(F,G) = 0$ does not imply that $F\isom G$.
Notice that if $F$ and $G$ are 0-interleaved, then $F \isom G$.
However $d(F,G) = 0$ only implies that $F$ and $G$ are $\eps$-interleaved for all $\eps>0$.
This does not imply that $F \isom G$. 
For an example, consider $F,G \in \RVec$ where $F=0$ and $G(a)$ is the ground field for $a=0$ but is otherwise 0.
However it does satisfy the other conditions of a metric, so it is an 
\emph{extended pseudometric}.

\begin{theorem} \label{thm:metric}
    The function $d$ defined above is an extended pseudometric on any subset of the class of $\cat{(\R,\leq)}$-indexed diagrams in $\D$.
\end{theorem}

To prove the theorem, we will need the following lemma, which shows that the set of $\eps$ for which two diagrams are $\eps$-interleaved form a ray. 

\begin{lemma} \label{lem:biggerInterleaving}
  If the $\cat{(\R,\leq)}$-indexed diagrams $F$ and $G$ are $\eps$-interleaved, then they are also $\eps'$-interleaved for any $\eps' \geq \eps$.
\end{lemma}

\begin{proof}
  Let $\varphi: F \To GT_{\eps}$ and $\psi: G \To FT_{\eps}$ such that $(\psi T_{\eps}) \varphi = F \eta_{2\eps}$ and $(\varphi T_{\eps}) \psi = G \eta_{2\eps}$.

Let $\eps' \geq \eps$ and set $\bar{\eps} = \eps' - \eps$.  Recall that we have the natural transformation, $\eta_{\bar{\eps}}: \Id_{\cat{(\R,\leq)}} \To T_{\bar{\eps}}$, and thus $\eta_{\bar{\eps}} T_{\eps}: T_{\eps} \To T_{\bar{\eps}} T_{\eps} = T_{\eps'}$. 
  Therefore $G\eta_{\bar{\eps}} T_{\eps}: GT_{\eps} \To GT_{\eps'}$.
  Define $\hat{\varphi} = (G \eta_{\bar{\eps}} T_{\eps}) \varphi$.
  For example,
  \begin{equation*}
    \hat{\varphi}(a): F(a) \xto{\varphi(a)} G(a+\eps) \xto{G\eta_{\bar{\eps}} T_{\eps} (a)} G(a + \eps').
  \end{equation*}
Similarly, define $\hat{\psi} = (F \eta_{\bar{\eps}} T_{\eps}) \psi$.

We see that $(\hat{\psi} T_{\eps'}) \hat{\varphi} = F \eta_{2\eps'}$ from the following commutative diagram.
\begin{equation*}
  \xymatrix@C=0.5em{
    F(a) \ar[dr]_{\varphi a} \ar[rr]^{F\eta_{2\eps}a} & & F(a+2\eps) \ar[rr]^{F\eta_{\bar{\eps}}T_{2\eps}a} & & F(a+\eps'+\eps) \ar[rr]^{F \eta_{\bar{\eps}} T_{\eps+\eps'}a} & & F(a+2\eps') \\
  & G(a+\eps) \ar[ur]_{\psi T_{\eps} a} \ar[rr]_{G \eta_{\bar{\eps}} T_{\eps} a} & & G(a+\eps') \ar[ur]_{\psi T_{\eps'} a}
  }
\end{equation*}
Similarly, one can check that $(\hat{\varphi} T_{\eps'}) \hat{\psi} = G \eta_{2\eps'}$.
\end{proof}

\begin{proof}[Proof of Theorem~\ref{thm:metric}]
  The identity natural transformation shows that $d(F,F) = 0$ for any diagram $F$.
By the symmetry of the definition of $\eps$-interleaving, we see that $d(F,G)=d(G,F)$ for any diagrams $F$ and $G$.
It remains to show the triangle inequality.

Consider diagrams $F$, $G$, and $H$.
Let $a = d(F,G)$ and $b = d(G,H)$. Let $\eps>0$.
Then by Lemma~\ref{lem:biggerInterleaving} and the definition of infimum, $F$ and $G$ are $(a+\eps)$-interleaved and $G$ and $H$ are $(b+\eps)$-interleaved.
Let $\varphi': F \To G T_{a+\eps}$ and $\psi': G \To F T_{a+\eps}$
and $\varphi'': G \To H T_{b+\eps}$ and $\psi'': H \To G T_{b+\eps}$
be the corresponding natural transformations.
We will show that composing these natural transformations gives the desired natural transformations for an interleaving of $F$ and $H$.

Let $\varphi = (\varphi'' T_{a + \eps}) \varphi': F \To H T_{b+\eps} T_{a+\eps} = H T_{a+b+2\eps}$ and $\psi = (\psi' T_{b+\eps}) \psi'': H \To F T_{a+\eps} T_{b+\eps} = F T_{a+b+2\eps}$.
The first composition comes from the following diagram. The second is similar.
\begin{equation*}
    \xymatrix{
      \cat{(\R,\leq)} \ar[r]^{T_{a+\eps}} \ar[d]_F  \ar@<1ex>@{}[dr]|*+{\stackrel{\mathlarger{\varphi'}}{\Rightarrow}} 
& \cat{(\R,\leq)} \ar[d]_G \ar[r]^{T_{b+\eps}} \ar@<1ex>@{}[dr]|*+{\stackrel{\mathlarger{\varphi''}}{\Rightarrow}} 
& \cat{(\R,\leq)} \ar[d]^H \\ 
D \ar@{=}[r] & D \ar@{=}[r] & D
    }
\end{equation*}

We claim that $(\psi T_{a+b+2\eps}) \varphi = F \eta_{2(a+b+2\eps)}$ and $(\varphi T_{a+b+2\eps}) \psi = H \eta_{2(a+b+2\eps)}$.
The first identity comes from the following diagram. The second is similar.
\begin{equation*}
    \xymatrix{
      \cat{(\R,\leq)} \ar[r]^{T_{a+\eps}} \ar[d]_F 
      \ar@<1ex>@{}[dr]|*+{\stackrel{\mathlarger{\varphi'}}{\Rightarrow}} &
   \cat{(\R,\leq)} \ar[r]^{T_{b+\eps}} \ar[d]_G      
      \ar@<1ex>@{}[dr]|*+{\stackrel{\mathlarger{\varphi''}}{\Rightarrow}} &
     \cat{(\R,\leq)} \ar[d]_H \ar[r]^{T_{b+\eps}}       
      \ar@<1ex>@{}[dr]|*+{\stackrel{\mathlarger{\psi''}}{\Rightarrow}} &
     \cat{(\R,\leq)} \ar[d]^G \ar[r]^{T_{a+\eps}}       
      \ar@<1ex>@{}[dr]|*+{\stackrel{\mathlarger{\psi'}}{\Rightarrow}} &
     \cat{(\R,\leq)} \ar[d]^F \\ 
   D \ar@{=}[r] & D \ar@{=}[r] & D \ar@{=}[r] & D \ar@{=}[r] & D
    }
\end{equation*}

Thus $F$ and $H$ are $(a+b+2\eps)$-interleaved for all $\eps > 0$.
Therefore $d(F,H) \leq a+b$.
\end{proof}

Let us declare $F$ equivalent to $G$ if $d(F,G)=0$; this is an equivalence relation, and we obtain the following corollary.

\begin{corollary} \label{cor:metric}
  If we identify diagrams whose interleaving distance is $0$, then $d$ is an extended metric on this set of equivalence classes.
\end{corollary}

One of the mostly useful aspects of the categorical view of interleavings is that if we apply a functor to $\eps$-interleaved diagrams, then the resulting diagrams are also $\eps$-interleaved. That is,

  \begin{proposition} \label{prop:functoriality}
    Let $F,G: \cat{(\R,\leq)} \to \D$ and $H:\D \to \E$. 
    If $F$ and $G$ are $\eps$-interleaved, then so are $HF$ and $HG$. Thus,
    \begin{equation*}
      d(HF, HG ) \leq d(F,G).
    \end{equation*}
  \end{proposition}

 \begin{proof}
    Assume $F$ and $G$ are $\eps$-interleaved. Let $\varphi: F \To G T_{\eps}$, $\psi: G \To F T_{\eps}$ be the corresponding natural transformations. Then by functoriality, $H \varphi: HF \To HGT_{\eps}$ and $H \psi: HG \To HFT_{\eps}$, and 
$(H\psi T_{\eps})(H\varphi) = (HF) \eta_{2\eps}$ and $(H\varphi T_{\eps})(H\psi) = (HG) \eta_{2\eps}$, as pictured in the following diagram.
\begin{equation*}
    \xymatrix{
      \cat{(\R,\leq)} \ar[r]^{T_{\eps}} \ar[d]_F             \ar@<1ex>@{}[dr]|*+{\stackrel{\mathlarger{\varphi}}{\Rightarrow}} 
& \cat{(\R,\leq)} \ar[d]_G \ar[r]^{T_{\eps}}       \ar@<1ex>@{}[dr]|*+{\stackrel{\mathlarger{\psi}}{\Rightarrow}} 
 &  \cat{(\R,\leq)} \ar[d]^F \\ 
   D \ar[d]_H \ar@{=}[r] \ar@{}[dr]|*+{=} 
&   D \ar[d]_H \ar@{=}[r] \ar@{}[dr]|*+{=} 
& D\ar[d]_H \\
   E \ar@{=}[r] &    E \ar@{=}[r] & E
   }
\end{equation*}
Therefore $HF$ and $HG$ are $\eps$-interleaved.
\end{proof}

\section{Diagrams of vector spaces}
\label{sec:vec}

From our categorical point of view, persistent homology calculations are done on diagrams in the category $\Vec$ of finite-dimensional vector spaces over a fixed ground field $\F$.
In this section we study $(\R,\leq)$-indexed diagrams in $\Vec$, and define some of the usual characters in topological persistence in this setting: barcodes, persistence diagrams, and the bottleneck distance.
Our main result is an isometric embedding of the set of finite barcodes with the bottleneck distance into the set of objects of $\RVec$ with the interleaving distance.

\medskip

The category $\Vec$ is one of the motivating examples of an abelian category. If the target category in a diagram category is an abelian category, then the diagram category inherits this structure. The necessary constructions are done objectwise. In particular, $\RVec$ is an abelian category.

\subsection{Finite type diagrams}

In this section we define \emph{finite type} and \emph{tame} diagrams in $\RVec$ and show that the two conditions are equivalent.
As a corollary, we obtain a Krull-Schmidt theorem.

\begin{definition} \label{def:finiteType}
  Given an interval $I \subseteq \R$, define the diagram $\chi_I \in \RVec$ by
  \[ \chi_I(a) =\begin{cases} \F& \text{if $a \in I$},\\ 0& \text{otherwise},
  \end{cases} \quad \quad \quad
      \chi_I(a\leq b) =\begin{cases} \Id_{\F}& \text{if $a,b \in I$},\\ 0& \text{otherwise}.
  \end{cases} 
  \]
  Say that a diagram $F \in \RVec$ has \emph{finite type} if $F \isom \oplus_{k=1}^N \chi_{I_k}$.
\end{definition}

We remark that $\chi_{\R}$ and $\chi_{\emptyset}$ are the constant functors $\F$ and $0$ respectively.

\begin{lemma}
  For an interval $I \subseteq \R$, the diagram $\chi_I$ is indecomposable.
\end{lemma}

\begin{proof}
  Assume that $\chi_I \isom P \oplus Q$. 
  If there is some $c \notin I$, then $P(c) \oplus Q(c) \isom \chi_I(c) = 0$, and therefore $P(c) = Q(c) = 0$.

  Let $a \in I$. Then $P(a) \oplus Q(a) \isom \chi_I(a) \isom \F$.
  Without loss of generality, assume that $P(a) \isom \F$ and $Q(a) = 0$.
  Let $a \leq b \in I$.
  Since $Q(a) = 0$, $Q(a\leq b) = 0$. 
  Thus it follows from $P(a\leq b) \oplus Q(a\leq b) = (P\oplus Q)(a\leq b) \isom \chi_I(a\leq b) = \Id_{\F}$, that $P(a\leq b) \isom \Id_{\F}$.
  Hence from $P(b) \oplus Q(b) \isom \chi_I(b) \isom \F$ we get that $P(b) \isom \F$ and $Q(b) = 0$. Similarly, if $d \leq a \in I$, we get that $Q(d\leq a) = 0$, $P(d\leq a) \isom \Id_{\F}$, $P(d) \isom \F$ and $Q(d)=0$.

We have shown that $P \isom \chi_I$ and $Q = 0$.
Therefore $\chi_I$ is indecomposable.
\end{proof}

The following definitions are variations of those in~\cite{cseh:stability}.

\begin{definition}
  Let $F \in \RVec$. Let $I \subseteq \R$ be an interval.
  Say that $F$ is \emph{constant on $I$}, if for all $a \leq b \in I$, $F(a\leq b)$ is an isomorphism.
Call $a \in \R$ a \emph{regular value} of $F$, if there is some open interval $I \ni a$ such that $F$ is constant on $I$.
Otherwise call $a$ a \emph{critical value} of $F$.\footnote{Even if $F$ is induced by sublevel sets, it is inadequate to define $a \in \R$ to be a critical value of $F$ 
 if for all sufficiently small $\eps > 0$ the map $F(a-\eps \leq a+\eps)$ is not an isomorphism~\cite{cseh:stability}.
Consider the example $X = \{(x,y) \in \R^2 \st 0\leq x \leq 1, \ 0 < y \leq 1\}$ and $f(x,y)$ equals 0 if $x=0$, $-1$ if $x=1$, and is otherwise equal to $y$.
Then 0 is not a critical value under this stricter definition, but the map $H_0(f^{-1}(-\infty,0]) \to H_0(f^{-1}(-\infty,1])$ induced by inclusion is not an isomorphism, contradicting the Critical Value Lemma.}
Call $F$ \emph{tame} if it has a finite number of critical values.
\end{definition}

\begin{lemma}[Critical Value Lemma] \label{lem:regularInterval}
  If an interval $I$ does not contain any critical values of $F$, then $F$ is constant on $I$.
\end{lemma}

\begin{proof}
  Let $a \leq b \in I$. By assumption, for all $c \in [a,b]$, there exists an interval $I_c \ni c$ such that $F$ is constant on $I_c$. Since $[a,b]$ is compact, the cover $\{I_c \st c \in [a,b]\}$ has a finite subcover $\{I_{c_1}, \ldots, I_{c_n}\}$.
Choose a sequence $a = d_0 \leq d_1 \leq \cdots \leq d_{m+1} = b$ such that for all $0 \leq k \leq m$, $d_k,d_{k+1} \in I_{c_j}$ for some $1\leq j \leq n$.
Then $F(d_k \leq d_{k+1})$ is an isomorphism for $0 \leq k \leq m$ and thus $F(a\leq b)$ is an isomorphism.
\end{proof}

We will need the following lemma in the proof of Theorem~\ref{thm:tameFiniteType}.
We refer the reader to Section~\ref{sec:algebra} for the definition of a finite type graded $\F[t]$-module. 

\begin{lemma} \label{lem:ZVec}
  The category $\ZVec$ is isomorphic to the category of finite type graded $\F[t]$-modules.
\end{lemma}

\begin{proof}
  To each diagram $F \in \ZVec$, we can assign the finite type graded $\F[t]$-module $M$, where for $k \in \Z$, $M_k = F(k)$ and for $a \in M_k$, $t \cdot a = F(k\leq k+1)(a)$.

  To each finite type graded $\F[t]$-module $M$, we can assign the diagram $F \in \ZVec$ given by $F(k) = M_k$ and whose morphisms are generated by $F(k \leq k+1)(a) = t\cdot a$ for $a \in F(k)$.

Both composites of these two functors are equal to the identity functor.
\end{proof}


\begin{theorem}
  \label{thm:tameFiniteType}
  A diagram in $\RVec$ is tame if and only if it has finite type.
\end{theorem}

\begin{proof}
To prove the `if' statement, we consider an interval $I \subseteq \R$. By definition, $a \in \R$ is a critical value of $\chi_I$ if and only if $a$ is an endpoint of $I$.
Let $F \in \RVec$ such that $F \isom \oplus_{k=1}^N \chi_{I_k}$. 
Then $a \in \R$ is a critical value of $F$ if and only if it is an endpoint of one of the intervals $I_k$, and so $F$ is tame.

The remainder of the proof is devoted to establishing the `only if' statement. Assume $F \in \RVec$ has critical values $a_1 < a_2 < \cdots < a_n$.
Choose $b_0,\ldots, b_n$ such that $b_0 < a_1$, for $k \in \{1,\ldots,n-1\}$, $a_k < b_k < a_{k+1}$ and $a_n < b_n$. For convenience, set $a_0 = -\infty$, $a_{n+1} = \infty$ and $F(a_{0})=F(b_{0})$, $F(a_{n+1}) = F(b_n)$. We have the ordered sequence,
\begin{equation*}
  -\infty = a_0 < b_0 < a_1 < b_1 < a_2 < \cdots < b_{n-1} < a_n < b_n < a_{n+1} = \infty.
\end{equation*}
We now identify the finite-valued part of this sequence with the integers from 0 to $2n$.

More precisely, we define a functor $i:[2n] \to \cat{(\R,\leq)}$ given by 
\begin{equation*}
  k \mapsto
  \begin{cases}
    b_{\frac{k}{2}} & \text{if } k \text{ is even} \\
    a_{\frac{k+1}{2}} & \text{if } k \text{ is odd.}
  \end{cases}
\end{equation*}
We also define a functor $r: \cat{(\R,\leq)} \to [2n]$ given by
\begin{equation*}
  c \mapsto
  \begin{cases}
    2k-1 & \text{if } c=a_k, \ k \in \{1,\ldots,n\} \\
    2k & \text{if } a_k < c < a_{k+1}, \ k \in \{0,\ldots,n\}.
  \end{cases}
\end{equation*}
Then we have the composite functor $ir: \cat{(\R,\leq)} \to \cat{(\R,\leq)}$ given by
\begin{equation*}
  c \mapsto
  \begin{cases}
    a_k & \text{if } c=a_k, \ k \in \{1,\ldots,n\} \\
    b_k & \text{if } a_k < c < a_{k+1}, \ k \in \{0,\ldots,n\}.
  \end{cases}
\end{equation*}
Precomposing $F$ with this functor gives us an induced functor $(ir)^*F \in \RVec$. That is, $(ir)^*F(c) = F(irc)$\footnote{Our notation comes from category theory; an arrow $f : x \rightarrow y$ defines a natural map $f^{*}:\operatorname{Hom}(y,z) \rightarrow \operatorname{Hom}(x,z)$ obtained by precomposing a given arrow $y \rightarrow z$ with $f$.}.
Notice that by Lemma~\ref{lem:regularInterval}, $F(irc) \isom  F(c)$ for all $c \in \R$. Thus, $(ir)^*:F \mapsto (ir)^*F$ is a natural isomorphism.

Next, $i^*F:[2n] \to \Vec$ can be extended to a functor $i^*F: {(\Z,\leq)} \to \Vec$ by setting $(i^*F)(k) = (i^{*}F)(0)$ for $k<0$, with $i^{*}F(k \leq 0)$ the identity, and for $k>2n$ setting $(i^*F)(k) = (i^*F)(2n)$ with $i^*F(2n\leq k)$ the identity.
By Lemma~\ref{lem:ZVec},
we can consider $i^*F$ to be a graded $\F[t]$-module.
Note that by assumption, $i^*F$ is a finitely-generated graded $\F[t]$-module.

By the structure theorem for finitely-generated graded modules over a
principal ideal domain (Theorem~\ref{thm:struct-pid}), there is a unique decomposition,
\begin{equation*}
  i^*F \isom \left( \bigoplus_{i=1}^{n_1} t^{c_i} \F[t] \right) \oplus 
\left( \bigoplus_{j=1}^{n_2} t^{d_j} \! \left( \F[t] /  (t^{e_j}) \right) \right).
\end{equation*}
It follows that as elements of $\ZVec$,
\begin{equation*}
  i^*F \isom \left( \bigoplus_{i=1}^{n_1} \chi_{[c_i,\infty)} \right) \oplus
  \left( \bigoplus_{j=1}^{n_2} \chi_{[d_j,d_j+e_j)} \right).
\end{equation*}
Therefore,
\begin{equation*}
  F \isom (ir)^*F = r^*i^*F \isom \left( \bigoplus_{i=1}^{n_1} r^*\chi_{[c_i,\infty)} \right) \oplus
  \left( \bigoplus_{j=1}^{n_2} r^*\chi_{[d_j,d_j+e_j)} \right),
\end{equation*}
where
\begin{equation*}
    r^*\chi_{[k,\infty)} = 
  \begin{cases}
    \chi_{[a_{\frac{k+1}{2}},\infty)} & \text{if $k$ odd}, \\
    \chi_{(a_{\frac{k}{2}},\infty)} & \text{if $k$ even},
 \end{cases}
\end{equation*}
and
\begin{equation*}
  r^*\chi_{[k,\ell)} =
    \begin{cases}
      \chi_{[a_{\frac{k+1}{2}},a_{\frac{\ell+1}{2}})} & \text{if } k,\ell \text{ odd}, \\
      \chi_{[a_{\frac{k+1}{2}},a_{\frac{\ell}{2}}]} & \text{if $k$ odd, $\ell$ even}, \\
      \chi_{(a_{\frac{k}{2}},a_{\frac{\ell+1}{2}})} & \text{if $k$ even, $\ell$ odd}, \\
      \chi_{(a_{\frac{k}{2}},a_{\frac{\ell+1}{2}}]} & \text{if $k, l$ even}.  
 \end{cases}  
\end{equation*} 
Thus $F$ has finite type.
\end{proof}

By the uniqueness of the decomposition in the structure theorem for graded modules over a graded PID in the previous proof, we get that finite type diagrams in $\RVec$ satisfy the following Krull--Schmidt theorem.
Compare this with~\cite[Proposition 2.2]{carlssonDeSilva:zigzag}.

\begin{corollary}[Krull--Schmidt] \label{cor:krullSchmidt}
  If $F \isom \oplus_{k=1}^n \chi_{I_k}$ and $F \isom \oplus_{j=1}^m \chi_{I'_j}$ then $n=m$ and the sequences $I_1,\ldots,I_n$ and $I'_1,\ldots,I'_m$ are the same up to reordering. \hfill \qedsymbol
\end{corollary}

\subsection{Barcodes and persistence diagrams}

Here we define \emph{barcodes} and \emph{persistence diagrams} for finite type diagrams in $\RVec$.
We observe that finite type diagrams in $\RVec$ are a categorification of finite barcodes.

\begin{definition}
  Assume $F \in \RVec$ has finite type. A \emph{barcode} is a multiset
  of intervals. The \emph{barcode} of $F$ is the multiset
  $\{I_k\}_{k=1}^n$ where $F \isom \oplus_{k=1}^n \chi_{I_k}$. This is
  well-defined by Corollary~\ref{cor:krullSchmidt} which follows from Theorem~\ref{thm:tameFiniteType}.

A \emph{persistence diagram} is a multiset of increasing pairs of extended real numbers. The \emph{persistence diagram} of $F$ is the multiset $\{(a_k,b_k)\}_{k=1}^n$, where $a_k \leq b_k$ and $\{a_k,b_k\}$ are the endpoints of $I_k$, with $F \isom \oplus_{k=1}^n \chi_{I_k}$.
Again, this is well-defined by Corollary~\ref{cor:krullSchmidt}.
\end{definition}

By Corollary~\ref{cor:krullSchmidt}, we immediately have the following. Compare with \cite[Corollary 3.1]{zomorodianCarlsson:computingPH} and \cite[Persistence Equivalence Theorem]{edelsbrunnerHarer:book}.
Note that a finite barcode is a finite multiset of intervals, not a multiset of finite intervals.

\begin{corollary}[Categorification of barcodes]
There is a bijection between isomorphism classes of finite type diagrams in $\RVec$ and finite barcodes. 
\end{corollary}

\subsection{Bottleneck distance}

In this section we define the \emph{bottleneck distance} between two barcodes in terms of the interleaving distance. We show that this results in the usual definition of~\cite{cseh:stability}.
We end by proving an isometric embedding of the set of finite barcodes with the bottleneck distance into the set of $\cat{(\R,\leq)}$-indexed diagrams in $\Vec$ with the interleaving distance.

\begin{definition}
  Given multisets $A$ and $B$, define the multiset $A_B$ to be the disjoint union of $A$ and the multiset containing the empty interval $\emptyset$ with cardinality $\abs{B}$.
  A \emph{stable bijection} or \emph{partial matching} between two multisets $A$ and $B$ is a bijection, $f: A_B \to B_A$. 
Write $f: A \rightleftharpoons B$.
\end{definition}

\begin{definition}
  Let $B$ and $B'$ be two barcodes. Define the \emph{bottleneck distance} between $B$ and $B'$ by
  \begin{equation} \label{eq:bottleneck}
    d_B(B,B') = \inf_{f: B \rightleftharpoons B'} \sup_{I \in \dom{f}} d(\chi_I,\chi_{f(I)}).
  \end{equation}
\end{definition}

On the right hand side of \eqref{eq:bottleneck} we have the interleaving distance. 
It follows from the following two propositions that this definition of bottleneck distance is equivalent to that in~\cite{cseh:stability}.

\begin{proposition} \label{prop:distance-intervals}
Let $I$ and $I'$ be two finite intervals.
\begin{enumerate}
\item If $I = I' = \emptyset$, then $d(\chi_I,\chi_{I'}) = 0$.
\item If $I' = \emptyset$ and $I$ has endpoints $a$ and $b$, then $d(\chi_I,\chi_{I'}) = \frac{b-a}{2}$.
\item If $I$ and $I'$ have endpoints $a,b$ and $a',b'$, respectively, then
\begin{equation*} \label{eq:distance-intervals}
  d(\chi_I,\chi_{I'}) = \min \left( \max(\abs{a-a'},\abs{b-b'}), \max \left(\frac{b-a}{2},\frac{b'-a'}{2} \right) \right).
\end{equation*}
\end{enumerate}
\end{proposition}

\begin{proposition} \label{prop:distance-intervals-infinite}
  Let $I$ and $I'$ be two intervals, at least one of which is infinite.
  \begin{enumerate}
  \item If $I = I' = \R$, then $d(\chi_I,\chi_{I'}) = 0$.
  \item If $\inf(I) = \inf(I') = -\infty$ and $I$ and $I'$ have right endpoints $b$ and $b'$, then $d(\chi_I,\chi_{I'}) = \abs{b-b'}$.
  \item If $\sup(I) = \sup(I') = \infty$ and $I$ and $I'$ have left endpoints $a$ and $a'$, then $d(\chi_I,\chi_{I'}) = \abs{a-a'}$.
  \item In all other cases, $d(\chi_I,\chi_{I'}) = \infty$.
 \end{enumerate}
\end{proposition}

Propositions~\ref{prop:distance-intervals} and \ref{prop:distance-intervals-infinite} follow from the following two lemmas. The proofs are technical yet straightforward, and we leave them to the motivated reader.

Assume the intervals $I$ and $I'$ are finite.
Let $h$ and $h'$ each denote half the length of the interval $I$ and $I'$, respectively, where the length of $\emptyset$ is 0.
If $I$ and $I'$ are nonempty, let $m$ and $m'$ denote their respective midpoints.

\begin{lemma} \label{lem:dleq}
  Assume $I$ and $I'$ are finite intervals.
  $d(\chi_I,\chi_{I'}) \leq \max(h,h').$
\end{lemma}

\begin{proof}
  Let $\eps > \max(h,h').$ Then $\chi_I \eta_{2\eps} = 0 = \chi_{I'} \eta_{2\eps}$.
  Let $\varphi = 0$ and $\psi = 0$.
  Then $\varphi$ and $\psi$ give an $\eps$-interleaving of $\chi_I$ and $\chi_{I'}$.
\end{proof}

\begin{lemma} \label{lem:dgeq}
  Assume $I$ and $I'$ are finite intervals.
  If $m \notin I'$, then $d(\chi_I,\chi_{I'}) \geq h$.
\end{lemma}

\begin{proof}
  Let $\eps < h$.
  Then $[m-\eps,m+\eps] \subset I$.
  Thus $\chi_I \eta_{2\eps}(m-\eps) = \Id_{\F}$.
  Suppose $m \notin I'$. 
  Assume there exists an $\eps$-interleaving ($\varphi$, $\psi$) of $\chi_I$ and $\chi_{I'}$.
  Then $(\psi T_{\eps}) \varphi(m-\eps) = \Id_{\F}$.
  But $\varphi(m-\eps) \in \chi_{I'}(m) = 0$.
  Therefore $(\psi T_{\eps}) \varphi(m-\eps) = 0$, which is a contradiction.
  Thus $d(\chi_I,\chi_{I'}) \geq h$.
\end{proof}

In the statement of the following theorem we abuse notation slightly by using $\RVec$ to denote the set of objects in the category $\RVec$.

\begin{theorem}[Categorification of the metric space of persistence diagrams]
\label{thm:embedding}
Let $\B$ be the set of finite barcodes, $d_B$ the bottleneck distance, and $d$ the interleaving distance. 
  The mapping $\chi$ defined by $\chi(\{I_k\}_{k=1}^n) = \oplus_{k=1}^n \chi_{I_k}$ gives an isometric embedding of metric spaces
  \begin{equation*}
    \chi: (\B,d_B) \incl (\RVec,d).
  \end{equation*}
\end{theorem}

\begin{proof}
  Let $B, B' \in \B$.
  By~\cite[Theorem 4.4]{ccsggo:interleaving}, we know that $d_B(B,B') \leq d(\chi(B),\chi(B'))$.
  It remains to show that $$d(\chi(B),\chi(B')) \leq d_B(B,B').$$ 
  If $d_B(B,B') = \infty$, then this is trivial. Assume that $d_B(B,B')<\infty$.

  Let $f:B \rightleftharpoons B'$ such that $\sup_{I \in \dom(f)} d(\chi_I,\chi_{f(I)}) < \infty$.
  Choose $\eps > \sup_{I \in \dom(f)} d(\chi_I,\chi_{f(I)})$.
  By Lemma~\ref{lem:biggerInterleaving}, 
  for each $I \in \dom(f)$, $\chi_I$ and $\chi_{f(I)}$ are $\eps$-interleaved.
  By Corollary~\ref{cor:directSumInterleaving}, $\chi(B)$ and $\chi(B')$ are $\eps$-interleaved.
 


  Thus $\chi(B)$ and $\chi(B')$ are $\eps$-interleaved for all $\eps>d_B(B,B')$.
  It follows that $d(\chi(B),\chi(B')) \leq d_B(B,B').$
\end{proof}

\section{Stability}
\label{sec:stability}

In~\cite{cseh:stability}, Cohen-Steiner, Edelsbrunner and Harer prove that persistent homology of sublevel sets of a function is stable with respect to perturbations of the function as measured by the supremum norm.
In this section, we use our categorical framework to generalize this Stability Theorem, as well as its generalization in~\cite{ccsggo:interleaving}.

\medskip

   Let $X \in \Top$. Assume $f,g: X \to \R$. Note that we do not require that $f$ and $g$ be continuous. 
Let $F \in \RTop$ be defined by $F(a) = f^{-1}(\infty,a]$ for $a \in \R$ and $F(a\leq b)$ is given by inclusion. Define $G$ similarly using $g$.
Let $H:\Top \to \D$ be any functor, e.g. singular homology with coefficients in a field $\F$, or rational homotopy groups.
Recall that $\norm{f-g}_{\infty} = \sup_{x \in X} \abs{f(x)-g(x)}$.

\begin{theorem}[Stability Theorem] \label{thm:stability}
     \begin{equation*}
        d(HF,HG) \leq ||f-g||_{\infty}.
      \end{equation*}
\end{theorem}

\begin{proof}
  Let $\eps = \norm{f-g}_{\infty}$.
    First we observe that by the assumption,
   \begin{equation*}
      F(a) = f^{-1}(-\infty,a] \subseteq g^{-1}(-\infty,a+\eps] = G(a+\eps).
    \end{equation*}
  and similarly, $G(a) \subseteq F(a+\eps)$.
  Thus, $F$ and $G$ are $\eps$-interleaved.
  It follows that $HF$ and $HG$ are $\eps$-interleaved (Proposition~\ref{prop:functoriality}), and thus
      \begin{equation*}
        d(HF,HG) \leq ||f-g||_{\infty}. \qedhere
      \end{equation*}
    \end{proof}

\section{Extended persistence}
\label{sec:extended}

In~\cite{cseh:extendingP}, Cohen-Steiner, Edelsbrunner and Harer, define extended persistence to obtain a sequence of vector spaces in which the homology classes of the total space do not live forever. Given a simplicial complex $K$ on $n$ ordered vertices, let $K_i$ be the subcomplex spanned by the first $i$ vertices, and let $L_i$ be the subcomplex spanned by the last $i$ vertices.
Let $H_k$ denote degree $k$ relative simplicial homology with coefficients in the field $\Z / 2\Z$.
Then they construct the sequence,
\begin{multline*}
  0 = H_k(K_0,\emptyset) \to H_k(K_1,\emptyset) \to \cdots \to H_k(K_n,\emptyset)\\ = H_k(K,L_0) \to H_k(K,L_1) \to \cdots \to H_k(K,L_n) = 0.
\end{multline*}
They show that the Stability Theorem of~\cite{cseh:stability} can be applied in this case.
Here we give a generalization of this construction and the corresponding stability theorem.

\medskip

Let $X \in \Top$. Assume $f: X \to \R$, where $f$ need not be continuous, and there exists an $M \in \R$ such that $f(x) \leq M$ for all $x \in X$.
Let $s > 0$ be an arbitrary amount to space out the upward and
downward filtrations.
Define the $\cat{(\R,\leq)}$-indexed diagram of pairs of topological spaces, $F \in \RPair$ as follows.

For $c < M+s$, let $F(c) = (f^{-1}(-\infty,c], \emptyset)$.
For $c \geq M+s$, let $F(c) = (X,f^{-1}[2M+s-c,\infty))$.
Notice that 
for $M \leq c < M+s$, $F(c) = (X,\emptyset)$,
and $F(M+s) = (X,f^{-1}(M))$.

For $c \leq d$, $F(c\leq d)$ is given by inclusion. Indeed, if $c \leq d < M+s$, we have $(f^{-1}(-\infty,c],\emptyset) \subseteq (f^{-1}(-\infty,d],\emptyset)$, if $M+s \leq c \leq d$, then $(X,f^{-1}[2M+s-c,\infty)) \subseteq (X,f^{-1}[2M+s-d,\infty))$, and if $c < M+s \leq d$, then 
$(f^{-1}(-\infty,c],\emptyset) \subseteq (X,f^{-1}[2M+s-d,\infty))$.

In the special case that there exists an $m \in \R$ such that $f(x) \geq m$ for all $x \in X$, then
for $c < m$, $F(c) = (\emptyset, \emptyset)$, and
for $c \geq M+s+(M-m) = 2M+s-m$, $F(c) = (X,X)$.

Now assume that we also have another (not necessarily continuous) map $g:X \to (-\infty,M]$.
Define $G \in \RPair$ similarly.
Let $H: \Pair \to \D$ be any functor, e.g. relative homology with coefficients in some field $\F$.

\begin{theorem}[Stability theorem for extended persistence] \label{thm:stability-extended}
  \begin{equation*}
    d(HF,HG) \leq \norm{f-g}_{\infty}.
  \end{equation*}
\end{theorem}

\begin{proof}
  Let $\eps = \norm{f-g}_{\infty}$.
  Let $c \in \R$.
  Then by assumption, $(f^{-1}(-\infty,c],\emptyset) \subseteq (g^{-1}(-\infty,c+\eps],\emptyset)$,
  $(X,f^{-1}[2M+s-c,\infty)) \subseteq (X,g^{-1}[2M+s-(c+\eps),\infty))$, and
  $(f^{-1}(-\infty,c],\emptyset) \subseteq (X,g^{-1}[2M+s-(c+\eps),\infty))$.
  Also, by assumption, we have the same relations with $f$ and $g$ switched.
  Thus, $F$ and $G$ are $\eps$-interleaved.
  It follows that $HF$ and $HG$ are $\eps$-interleaved (Proposition~\ref{prop:functoriality}), and thus
  \begin{equation*}
    d(HF,HG) \leq d(F,G) \leq \norm{f-g}_{\infty}. \qedhere
  \end{equation*}
\end{proof}

\section{Abelian structure of interleavings}
\label{sec:abelian}

Let $\D$ be a category and 
let $\eps \geq 0$.   
In this section, we consider the category $\Inter{D}$ of $\eps$-interleavings of diagrams in $\RD$, which we define below.
We will show that this construction is functorial, and that if $\D$ is an abelian category, then so is $\Inter{D}$.
As a corollary we obtain stability theorems for kernels, images and cokernels in persistence and extended persistence.

\medskip

First, let us recall that the functor $T_{\eps} : \mathbb{R} \rightarrow \mathbb{R}$, $T_{\eps}(x) = x + \eps$, $T_{\eps}(x \leq y) = x+\eps \leq y + \eps$, comes equipped with a ``unit'' natural transformation, $\eta_{\eps} : \Id \Rightarrow T_{\eps}$, since $x \leq x + \eps$.  We will write $\eta_{\eps}^{2}$ for the iteration $\Id \Rightarrow T_{\eps} \Rightarrow T_{\eps}^{2}$.

\begin{definition} \label{def:Inter}
The objects of $\Inter{D}$ are $\eps$-interleavings, $(F,G,\varphi,\psi)$, where $\varphi:F \Rightarrow GT_{\eps}$, $\psi:G \Rightarrow FT_{\eps}$, such that $(\psi T_{\eps}) \varphi = F \eta_{\eps}^{2}$ and $(\varphi T_{\eps}) \psi = G \eta_{\eps}^{2}$ (Definition~\ref{def:interleaving}).
A morphism $(\alpha,\beta):(F,G,\varphi,\psi) \To (F',G',\varphi',\psi')$ consists of a pair of natural transformations, $\alpha : F \Rightarrow F'$ and $\beta : G \rightarrow G'$, such that the diagrams,
\[
    \xymatrix{
        F \ar[r]^{\varphi} \ar[d]_{\alpha} & G T_{\eps} \ar[d]^{\beta T_{\eps}} \\
        F' \ar[r]_{\varphi'} & G' T_{\eps}
    }
\quad
\text{and}
\quad
    \xymatrix{
        G \ar[r]^{\psi} \ar[d]_{\beta} & FT_{\eps} \ar[d]^{\alpha T_{\eps}} \\
        G' \ar[r]_{\psi'} & F'T_{\eps}
    }
\]
commute.
\end{definition}

Let us also verify the naturality of the above construction.

\begin{proposition} \label{prop:InterFunctor}
  Definition~\ref{def:Inter} of $\Inter[\eps]{D}$ is functorial in $\eps$ and in $\D$.
\end{proposition}

\begin{proof}
  Let $\eps \leq \eps'$. 
  The functor, $\Inter[\eps]{D} \to \Inter[\eps']{D}$, is defined on objects by Lemma~\ref{lem:biggerInterleaving}.
To be precise, we have $(F,G,\varphi,\psi) \mapsto (F,G,\hat{\varphi},\hat{\psi})$, where 
$\hat{\varphi} = (G\eta_{\eps'-\eps}T_{\eps})\varphi$, 
   and $\hat{\psi} = (F\eta_{\eps'-\eps}T_{\eps})\psi$. 
Let $(\alpha,\beta): (F,G,\varphi,\psi) \to (F',G',\varphi',\psi') \in \Inter{D}$.
Then the following commutative diagram
\begin{equation*}
  \xymatrix{F \ar[d]_{\alpha} \ar[r]^{\varphi} & GT_{\eps} \ar[d]^{\beta T_{\eps}}  \ar[r]^{G\eta_{\eps'-\eps}T_{\eps}} & GT_{\eps'} \ar[d]^{\beta T_{\eps'}} \\
    F' \ar[r]_{\varphi'} & G'T_{\eps} \ar[r]_{G'\eta_{\eps'-\eps}T_{\eps}} & G'T_{\eps'}
}
\end{equation*}
and a similar one show that $(\alpha,\beta): (F,G,\hat{\varphi},\hat{\psi}) \to (F',G',\hat{\varphi}',\hat{\psi}') \in \Inter[\eps']{D}$.
From this it follows that $\Inter[\eps]{D}$ is functorial in $\eps$.

Now consider a functor $H: \D \to \D'.$ The induced functor $\Inter{D} \to \Inter{D'}$ is defined by composition with $H$ (for the details of the definition on objects, see the proof of Proposition~\ref{prop:functoriality}). 
It follows that $\Inter[\eps]{D}$ is functorial in $\D$.
\end{proof}

Let $\A$ be an abelian category. Then, as discussed at the start of Section~\ref{sec:vec}, so is $\RA$. 
We claim that $\Inter{A}$ is also an abelian category.  Recall (Section~\ref{subsubsec:abcat}) that a category is \emph{abelian} if it has a zero object, all finite products and coproducts, every morphism has a kernel and a cokernel, and all monomorphisms and epimorphisms are kernels and cokernels, respectively.
 
\begin{lemma}\label{lem:inter-zero}
The category $\Inter{A}$ has a zero object.
\end{lemma}

\begin{proof}
The zero object of $\Inter{A}$ comes from the zero object, 0, of $\A$.  The diagram category $\cat{Vec}^{(\mathbb{R},\leq)}$ then inherits the constant zero diagram, $O_{x} = 0$ for all $x \in \mathbb{R}$, with the identity morphism $O_{x} \rightarrow O_{y}$ for $x \leq y$.  In turn, we define the trivial $\eps$-interleaving, $(O,O,\omega,\omega)$, where $\omega_{x} : O_{x} \rightarrow O_{x+\eps}$ is again the identity.  It is easy to see that $(O,O,\omega,\omega)$ is the desired zero object.  Indeed, to see that it is initial, we note that for every interleaving $(F,G,\varphi,\psi)$, and for every $x \in \mathbb{R}$, there are unique morphisms $O_{x} \rightarrow F_{x}$ and $O_{x} \rightarrow G_{x}$ because $O_{x}$ is initial in $\A$, and the appropriate diagrams commute.  Similarly, $(O,O,\omega,\omega)$ is final, and hence the desired zero object in $\Inter{A}$.  
\end{proof}

In particular, for any objects $X,Y \in \Inter{A}$, we now have the zero morphism $0 : X \rightarrow Y$, that is the composite $X \rightarrow O \rightarrow Y$.

\begin{lemma}\label{lem:inter-pbpo}
The category $\Inter{A}$ has all pull-backs and push-outs, and their components in $\RA$ are given by the respective pull-backs and push-outs in $\RA$.
\end{lemma}

\begin{proof}
We show that $\Inter{A}$ has all pull-backs. The arguments and constructions for push-outs are dual.

Consider the diagram
\[
	(F',G',\varphi',\psi') \xrightarrow{(\alpha',\beta')} (F,G,\varphi,\psi) \xleftarrow{(\alpha'',\beta'')} (F'',G'',\varphi'', \psi'').
\]
The category $\RA$
is abelian.  Thus we may form the pull-back functors,
\[
	\xymatrix{
		F' \times_{F} F'' \ar[r]^{\pi'_{F}} \ar[d]_{\pi''_{F}}
			& F'	\ar[d]^{\alpha'}	\\
		F'' \ar[r]_{\alpha''} & F
	}
	\quad
	\text{and}
	\quad
	\xymatrix{
		G \times_{G} G'' \ar[r]^{\pi'_{G}} \ar[d]_{\pi''_{G}}	& G' \ar[d]^{\beta'}	\\
		G''	\ar[r]_{\alpha''}							& F.
	}
\]

To simplify the notation, let $T = T_{\eps}$ and $\eta = \eta_{\eps}$.  
Observe that
\[
	G'T \times_{GT} G''T = \left( G' \times_{G} G'' \right) T,
\]
and so from the universal property of pull-backs we obtain the natural transformation,
\[
	\Phi = \varphi' \times_{\varphi} \varphi'' : F' \times_{F} F''  \rightarrow \left(G' \times_{G} G''\right) T.
\]
Similarly, we get a natural transformation
\[
	\Psi = \psi' \times_{\psi} \psi'' : G' \times_{G} G'' \rightarrow \left( F' \times_{F} F'' \right) T.
\]
We need to check that $(F' \times_{F} F'', G' \times_{G} G'', \Phi, \Psi)$ is an $\eps$-interleaving that is indeed the relevant pull-back.

To see that we have an interleaving, it only remains to show that $(\Phi T)\Psi = (G' \times_{G} G'')\eta^2$ and $(\Psi T)\Phi = (F' \times_{F} F'')\eta^{2}$.  We prove the second identity. The verification of the first is symmetric.

We observe that $(\Psi T)\Phi$ is one morphism $F' \times_{F} F'' \rightarrow (F' \times_{F} F'')T^{2}$ that provides the unique dotted arrow (by the universal property of the pull-back) making the diagram
\[
    \xymatrix{
        F' \times_{F} F'' \ar[rrr]^{\pi'_F} \ar[ddd]_{\pi''_F} \ar@{.>}[dr] 
            & & & F' \ar[dl]_{(\psi' T)\varphi'}^{=F'\eta^{2}} \ar[ddd]^{\alpha'}    \\
        & (F'\times_{F} F'')T^{2} \ar[r]^-{\pi'_F T^2} \ar[d]_{\pi''_F T^2}
            & F'T^{2} \ar[d]^{\alpha' T^2} \\
        & F''T^{2} \ar[r]_{\alpha'' T^2}
            & FT^{2}    \\
        F'' \ar[rrr]_{\alpha''} \ar[ur]^{(\psi'' T)\varphi''}_{=F''\eta^{2}}
            & & & F \ar[ul]^{(\psi T)\varphi}_{=F\eta^{2}}
    }
\]
commute.  Since $(F' \times_{F} F'')\eta^{2}$ also fits the diagram, by uniqueness, $(\Psi T)\Phi = (F' \times_{F} F'')\eta^{2}$.
\end{proof}

\begin{corollary} \label{cor:finiteProducts}
The category $\Inter{A}$ has all finite products and coproducts. Every morphism in $\Inter{A}$ has a kernel and a cokernel.
\end{corollary}

\begin{proof}
Since $\Inter{A}$ has a terminal object and pull-backs, it has finite products.  Since the category has an initial object and push-outs, it has finite coproducts.  Since $\Inter{A}$ has a zero object and pull-backs, and the kernel of $(\alpha,\beta) : (F, G, \varphi, \psi) \rightarrow (F',G', \varphi', \psi')$ can be obtained by pulling back along the initial morphism $(O,O, \omega, \omega) \rightarrow (F',G', \varphi', \psi')$, every morphism has a kernel.  Similarly, every morphism has a cokernel since $\Inter{A}$ has a zero object and push-outs.
\end{proof}

It remains to show that every monomorphism is a kernel, and that every epimorphism is a cokernel.  Before doing so, we show that in a preadditive category, monomorphisms and epimorphisms are characterized by their trivial kernels and cokernels, respectively.  

\begin{lemma}\label{lem:mono-ker}\cite{mitchell}
Let $\cat{C}$ be a category with zero object, kernels and cokernels. If $f:X \rightarrow Y$ is a monomorphism in $\cat{C}$, then $\ker f = 0$.  Dually, if $f$ is an epimorphism, then $\coker f = 0$.
\end{lemma}

If the category is preadditive, then we have the following converse for Lemma~\ref{lem:mono-ker}.

\begin{lemma}\label{lem:ker-mono}
Let $\cat{C}$ be a preadditive category with zero object, kernels and cokernels.  If $f : X \rightarrow Y$ is a morphism in $\cat{C}$ with trivial kernel, then $f$ is a monomorphism.  Dually, if $f$ has trivial cokernel, then $f$ is an epimorphism.
\end{lemma}

\begin{proof}
Suppose that $\ker f = 0$, and that $g,h:W \rightarrow X$ are morphisms that satisfy $fg = fh$.  Then using the preadditive structure of $\cat{C}$, we find that $f(g-h)=0$, and so $g-h$ factors through $\ker f = 0$.  It follows that $g-h=0$, so $g=h$. Therefore $f$ is a monomorphism.

Again, the proof of the dual statement is dual.
\end{proof}

Next, we show that the above characterization of monomorphisms and epimorphisms applies to our setting.

\begin{lemma}\label{lem:inter-additive}
The category $\Inter{A}$ is preadditive.
\end{lemma}

\begin{proof}
  Let $X = (F,G,\varphi,\psi)$ and $X' = (F',G',\varphi,\psi)$ be
  $\eps$-interleavings. By definition, $\Inter{A}(X,Y) \subset
  \RA(F,F') \times \RA(G,G')$, which
  is itself an abelian group. One can readily verify that $(0,0) \in \Inter{A}(X,Y)$.  Since
  $\RA$ is preadditive, composition distributes over the
  addition of morphisms.  If $(\alpha,\beta),(\alpha',\beta') \in
  \Inter{A}(X,Y)$, then $\varphi'(\alpha + \alpha') = \varphi'\alpha +
  \varphi'\alpha' = (\beta T)\varphi + (\beta' T)\varphi =
  ((\beta+\beta')T)\varphi$, so $(\alpha+\alpha',\beta + \beta') \in
  \Inter{A}(X,Y)$.  Finally, we verify that if $(\alpha,\beta) \in
  \Inter{A}(X,Y)$, then so is its additive inverse. Since $\RA$ is preadditive, we have natural transformations $-\alpha$ and $-\beta$.  Since $0 = \varphi' 0 = \varphi' (\alpha +
  (-\alpha)) = \varphi'\alpha + \varphi'(-\alpha)$, it follows that
  $\varphi'(-\alpha) = - \varphi'\alpha$.  Similarly, $(-\beta)\varphi
  = - \beta \varphi$.  It follows that $(-\alpha,-\beta) \in
  \Inter{A}(X,Y)$.  Since $(\alpha,\beta)+(-\alpha,-\beta)=(0,0)$, $\Inter{A}(X,Y)$ is an abelian group.

Composition is bilinear since it is the restriction of composition in the additive category $\RA \times \RA$. 
\end{proof}

\begin{lemma}\label{lem:inter-ker}
In $\Inter{A}$, every monomorphism is a kernel, and every epimorphism is a cokernel.
\end{lemma}

\begin{proof}
Let $(\alpha,\beta) : (F,G,\varphi,\psi) \rightarrow (F', G', \varphi, \psi)$ be a monomorphism in $\Inter{A}$.  We will show that $(\alpha,\beta)$ is the kernel of the natural morphism 
\[
	\pi:(F', G', \varphi, \psi) \rightarrow \coker(\alpha,\beta).
\] 
First, we calculate $\ker(\alpha,\beta)$ in terms of $\ker\alpha$ and $\ker\beta$.  If we pull back the morphism $(\alpha,\beta): (F,G,\varphi,\psi) \rightarrow (F',G',\varphi',\psi')$ along the initial morphism $(O,O,\omega,\omega) \rightarrow (F',G',\varphi',\psi')$, we obtain the interleaving $(\ker\alpha, \ker\beta, \Phi, \Psi)$ constructed in the proof of Lemma~\ref{lem:inter-pbpo}.

Cokernels are obtained in a dual manner; we have that $\coker(\alpha,\beta) = (\coker\alpha,\coker\beta, \bar{\Phi}, \bar{\Psi})$.

By Lemma~\ref{lem:mono-ker}, every monomorphism has trivial kernel.  
Thus $\ker(\alpha,\beta) = (\ker\alpha,\ker\beta, \Phi, \Psi) = (O,O,\omega, \omega)$.  This means, in particular, that $\ker\alpha = \ker \beta = O$.  It follows from Lemmas~\ref{lem:ker-mono} and~\ref{lem:inter-additive} that $\alpha$ and $\beta$ are monomorphisms.  Since $\RA$ is abelian,  $\alpha$ is the kernel of the quotient map, $F' \rightarrow \coker\alpha$, and likewise $\beta$ is the kernel of $G' \rightarrow \coker\beta$.  It then follows that $(\alpha,\beta)$ is the kernel of the natural morphism, $(F',G',\Phi,\Psi) \rightarrow (\coker\alpha,\coker\beta,\bar{\Phi},\bar{\Psi})$.

The dual statement follows from the dual proof.
\end{proof}

Combining Lemma~\ref{lem:inter-zero}, Corollary~\ref{cor:finiteProducts} and Lemma~\ref{lem:inter-ker} we have the following. 

\begin{theorem} \label{thm:inter-abcat}
  Given an abelian category $\A$ and $\eps \geq 0$, the category $\Inter{A}$ of $\eps$-interleavings in $\A$ is an abelian category. \hfill \qedsymbol
\end{theorem}


From Theorem~\ref{thm:inter-abcat} and Lemma~\ref{lem:inter-pbpo}, we immediately have the following two applications.

\begin{corollary} \label{cor:directSumInterleaving}
If the two pairs of diagrams $(F,G)$ and $(F',G')$ in  $\RA$ are $\eps$-interleaved, then so is the pair $(F \oplus F', G \oplus G')$.
\hfill \qedsymbol
\end{corollary}

\begin{corollary} \label{cor:ker-im-coker-interleaving}
Let $(\alpha,\beta)$ be a morphism in $\Inter{A}$.
Then each of the following three pairs of diagrams in $\RA$ are $\eps$-interleaved: $(\ker \alpha, \ker \beta)$, $(\im \alpha, \im \beta)$, and $(\coker \alpha, \coker \beta)$.
\hfill \qedsymbol
\end{corollary}

As an application of Corollary~\ref{cor:ker-im-coker-interleaving}, we get the following generalization of the Stability Theorem of~\cite{csehm:kernels}.

\begin{theorem}[Stability theorem for kernels, images and cokernels]
\label{thm:stability-kernels}
  Let $h:Y \to X$ be a continuous map of topological spaces.
  Let $f,f': X \to \R$ and $g,g':Y \to \R$ be (not necessarily continuous) maps,
  such that
\begin{equation} \label{eq:ineq_fgh}
 \text{for all }y \in Y, \ fh(y) \leq g(y) \text{ and }f'h(y) \leq g'(y).
\end{equation}
Let $F \in \RTop$ be given by $F(a) = f^{-1}(-\infty,a]$ and inclusion.
Define $F'$, $G$, and $G'$ similarly.
By \eqref{eq:ineq_fgh}, $h$ induces maps $\alpha: G \to F$ and $\beta: G' \to F'$ in $\RTop$.
Let $\A$ be an abelian category and $H:\Top \to \A$ be some functor.
Let $\eps = \max\{ \norm{f-f'}_{\infty}, \norm{g-g'}_{\infty} \}$. 
Then 
\begin{equation*}
  d(\ker H\alpha, \ker H\beta), d(\im H\alpha, \im H\beta), d(\coker H\alpha, \coker H\beta) \leq \eps.
\end{equation*}
\end{theorem}

\begin{proof}
  By the definition of $\eps$, 
 $F \incl F'T_{\eps}$, $F' \incl FT_{\eps}$, $G \incl G'T_{\eps}$, and $G' \incl GT_{\eps}$ in $\RTop$. Since $\alpha$ and $\beta$ are both induced by $h$, the diagrams
  \begin{equation*}
    \xymatrix{G \ar@{^(->}[r]  \ar[d]_{\alpha} & G'T_{\eps} \ar[d]^{\beta T_{\eps}} \\
      F \ar@{^(->}[r] & F'T_{\eps}}
 \quad \text{and} \quad
   \xymatrix{G' \ar@{^(->}[r]  \ar[d]_{\beta} & GT_{\eps} \ar[d]^{\alpha T_{\eps}} \\
      F' \ar@{^(->}[r] & FT_{\eps}}
  \end{equation*}
commute, and thus $(\alpha,\beta) \in \Inter{\Top}$.
By functoriality $(H\alpha,H\beta) \in \Inter{A}$ (see Proposition~\ref{prop:InterFunctor}).
Apply Corollary~\ref{cor:ker-im-coker-interleaving} and Definition~\ref{def:interleaving-distance} to obtain the desired result.
\end{proof}

Strengthening \eqref{eq:ineq_fgh}, we obtain an extended persistence version of this theorem. It has essentially the same proof, so we omit it.

\begin{theorem}[Stability theorem for kernels, images and cokernels in extended persistence] \label{thm:stability-kernels-extended}
  Let $h:Y \to X$ be a continuous map of topological spaces.
  Let $f,f': X \to (-\infty,M]$ be (not necessarily continuous) maps.
  Let $g=fh$ and $g'=f'h$.
  Let $s > 0$.
Let $F \in \RPair$ be given by $F(c) = (f^{-1}(-\infty,c],\emptyset)$ if $c < b + s$ and $F(c) = (X,f^{-1}[2b+s-c,\infty))$ if $c \geq b+s$, and inclusion.
Define $F'$, $G$, and $G'$ similarly.
Then $h$ induces maps $\alpha: G \to F$ and $\beta: G' \to F'$ in $\RPair$.
Let $\A$ be an abelian category and $H:\Pair \to \A$ be some functor.
Then 
\begin{equation*}
  d(\ker H\alpha, \ker H\beta), d(\im H\alpha, \im H\beta), d(\coker H\alpha, \coker H\beta) \leq 
\norm{f-f'}_{\infty}.
\end{equation*}
\end{theorem}

\section{Future work}
\label{sec:future}

In this paper we have studied persistence by considering diagrams indexed by $\cat{(\R,\leq)}$. 
However there are versions of persistence in which the objects of study can be viewed as diagrams with more general indexing categories. For example, we would like to be able to consider diagrams indexed by $(\R^n,\leq)$ for multi-dimensional persistence, $S^1$ for circle-valued persistence, and the category $\cdot \to \cdot \from \cdot \to \cdots \from \cdot$ for zig-zag persistence. 
This generalization will be presented in~\cite{bdss:1}.

It would be nice to have a categorical definition of bottleneck
distance arbitrary $(\R,\leq)$-indexed diagrams of vector spaces and to
have a corresponding Isometry Theorem.

\section*{Acknowledgments}

The first author would like to thank Vin de Silva, Robert Ghrist,
David Lipsky, John Oprea and Radmila Sazdanovic for helpful
conversations. Both authors thank Peter Landweber for many helpful
corrections.
The first author also gratefully acknowledges the support from AFOSR grant
 FA9550-13-1-0115.



\newcommand{\etalchar}[1]{$^{#1}$}

\end{document}